\documentclass[a4paper,11pt]{article}
\usepackage{pdfsync}

 \usepackage[plainpages=false]{hyperref}
 \usepackage{amsfonts,amsmath,amssymb,amsthm}
 \usepackage{latexsym,lscape,rawfonts,mathrsfs}
\usepackage{enumitem}

 \usepackage[dvips]{color}
 \usepackage{multicol}

 \usepackage[all]{xy}
 \usepackage{eufrak}
 \usepackage{makeidx}
 \usepackage{graphicx,psfrag}
 \usepackage{pstool}

 \usepackage{array,tabularx}

 \usepackage{setspace}

 \usepackage[titletoc,title]{appendix}

\usepackage{txfonts}

 \newcommand{\ba}{\begin{align*}}
 \newcommand{\ea}{\end{align*}}
 \newcommand{\na}{\nabla}
\newcommand{\la}{\langle}
\newcommand{\ra}{\rangle}
\newcommand{\lc}{\left(}
\newcommand{\rc}{\right)}
\newcommand{\pt}{\partial_t}
\newcommand{\df}{\Delta_f}
\newcommand{\ep}{\epsilon}

 \newcommand{\dvol}{\text{d}V}
\newcommand{\dmu}{\text{d}\mathbf{\mu}}

 \newcommand{\area}{\mathcal{A}}

 \newcommand{\diam}{\text{diam}}

 \makeatletter
 \def\ExtendSymbol#1#2#3#4#5{\ext@arrow 0099{\arrowfill@#1#2#3}{#4}{#5}}
 
 \makeatother

 \makeatletter
 \def\ExtendSymbol#1#2#3#4#5{\ext@arrow 0099{\arrowfill@#1#2#3}{#4}{#5}}
 \newcommand\longright[2][]{\ExtendSymbol{-}{-}{\rightarrow}{#1}{#2}}
 \makeatother

\def\Xint#1{\mathchoice
{\XXint\displaystyle\textstyle{#1}}%
{\XXint\textstyle\scriptstyle{#1}}%
{\XXint\scriptstyle\scriptscriptstyle{#1}}%
{\XXint\scriptscriptstyle\scriptscriptstyle{#1}}%
\!\int}
\def\XXint#1#2#3{{\setbox0=\hbox{$#1{#2#3}{\int}$ }
\vcenter{\hbox{$#2#3$ }}\kern-.55\wd0}}

\def\aint{\Xint-}

\numberwithin{equation}{section}

\newtheorem{thm}{Theorem}[section]
\newtheorem{cor}[thm]{Corollary}
\newtheorem{prop}[thm]{Proposition}
\newtheorem{lem}[thm]{Lemma}

\newtheorem{defn}[thm]{Definition}

\theoremstyle{remark}
\newtheorem{remark}[thm]{\textbf{Remark}}

\title{On the regular-convexity of Ricci shrinker limit spaces}
\author{Shaosai Huang,\ Yu Li\ and\ Bing Wang}
\date{\today}

\begin{document}
\maketitle

\begin{abstract}
In this paper we study the structure of the pointed-Gromov-Hausdorff limits of
sequences of Ricci shrinkers. We define a regular-singular decomposition
following the work of Cheeger-Colding for manifolds with a uniform Ricci
curvature lower bound, and prove that the regular part of any non-collapsing
Ricci shrinker limit space is strongly convex, inspired by Colding-Naber's
original idea of parabolic smoothing of the distance functions.
\end{abstract}

\tableofcontents

\section{Introduction}
A \emph{Ricci shrinker} is a triple $(M,g,f)$ where $(M,g)$ is a complete
Riemannian manifold, and $f$ is a $C^2$ potential function on $M$ such that its
Ricci curvature $Rc$ satisfies
\begin{align} 
Rc+Hess_f\ =\ \frac{1}{2}g,
\label{E100}
\end{align}
where $f$ is normalized by adding a constant, if necessary, so that the scalar
curvature $R$ satisfies
\begin{align} 
R+|\nabla f|^2\ =\ f.
\label{E101}
\end{align}
We will always fix some minimal point $p\in M$ (whose existence guaranteed
by Lemma \ref{L100}) as a base point, making a pointed Ricci shrinker
$(M,p,g,f)$. We also recall the following fundamental fact (due to Binglong
Chen~\cite{BLC}) for the scalar curvature on a Ricci shrinker:
\begin{align}\label{eqn: positive_scalar}
R\ \ge\ 0.
\end{align} 

Ricci shrinkers, usually regarded as generalizations of positive Einstein manifolds,
form an important collection of objects for our understanding of the
singularities of Ricci flows. Indeed, Ricci shrinkers are critical points of
Perelman's $\boldsymbol{\mu}$-functional, see \cite{Pe1}. Up to dimension three, all Ricci
shrinkers are classified up to isometry, see \cite{Ham95}, \cite{Pe1},
\cite{Pe2}, \cite{Naber10}, \cite{NW08}, and \cite{CCZ08}. However, the higher,
even four, dimensional cases are much more complicated and a usual approach is
to consider the whole collection of a given dimension as a moduli space.
Important questions immediately arise: Is this moduli space compact with
respect to some reasonable topology? If not, is there a standard model for the
added points in the moduli space as a result of compactification? Systematic
studies of the moduli space of complete Ricci shrinkers are initiated
in~\cite{LLW17}, where the above questions are partially answered: it is shown
that a sequence of non-collapsed smooth Ricci shrinkers is expected to
subconverge, in the pointed-$\hat{C}^{\infty}$-Cheeger-Gromov topology, to a
metric space called a \emph{conifold Ricci shrinker}, see \cite[Theorem
8.6]{LLW17}.

There is yet another nice property that a conifold Ricci shrinker is defined to
satisfy: The regular part  $\mathcal{R}$ of the limit space $X$, which is an
open manifold, should be strongly convex relative to $X$. Here we say that
$\mathcal{R}$ is \emph{strongly convex} if any limit minimal  geodesic
intersecting $\mathcal{R}$ non-trivially has its entire interior contained
in $\mathcal{R}$. Note that this is a slightly stronger concept compared to the
usual geodesic convexity.

The main purpose of the current paper is then to prove the desired strong
convexity of the regular part, therefore justifying the limit
space to be indeed a conifold Ricci shrinker.
Let us denote by $\mathcal{M}_m(A)$ the moduli space of $m$-dimensional Ricci
shrinkers with a uniform $\boldsymbol{\mu}$-entropy lower bound by $-A$ ($A>0$
fixed), and our first result is the following regular-convexity theorem:
\begin{thm}[Regular-convexity of Ricci shrinker limits]
Let $\{(M_i,p_i,g_i,f_i)\}\subset \mathcal{M}_m(A)$ be a sequence of Ricci
shrinkers that converges, in the pointed-$\hat{C}^{\infty}$-Cheeger-Gromov
topology, to a metric space with potential $(X,p,d,f)$, then $X$ has a
regular-singular decomposition such that the regular part is strongly convex.
Therefore, $(X,p,d,f)$ is a conifold Ricci shrinker.
\label{thm: main2}
\end{thm}

\begin{remark}
The concept of conifold Ricci shrinkers has its origin in the K\"ahler-Ricci
flat setting~\cite[Definition 1.2]{CW17A}, where the collection of certain
Calabi-Yau conifolds was shown to be compact  in the
pointed-$\hat{C}^{\infty}$-Cheeger-Gromov topology, and such compactness played
a fundamental role in the resolution of the Hamilton-Tian conjecture for
K\"ahler-Ricci flows in \cite{CW17B}.
\end{remark}

We refer the readers to \cite{LLW17} and Section 2.2 for detailed discussions of
the concepts involved. This theorem generalizes Colding-Naber's fundamental
regular-convexity theorem in~\cite{CN12} when the Ricci curvature is uniformly
bounded from below. In view of the close relationship between the geometry of
Ricci shrinkers and of manifolds with a uniform Bakry-\'Emery Ricci curvature
lower bound, we pursue a similar path that leads to Colding-Naber's theorem.

An alternative approach to Theorem~\ref{thm: main2}, as one may suggest,
would be applying a suitable conformal transformation, so that the resulted
manifold will locally acquire Ricci curvature bounds (see \cite{ZLZ10},
\cite{HM14} and \cite[Lemma 3.7]{LLW17}), and Colding-Naber's original H\"older
continuity theorem could be directly applied on an increasing sequence of
exhausting domains to prove the global convexity result. The technique of taking
suitable conformal transformations has actually be utilized in \cite{LLW17} to
improve the regularity of convergence. But it fails in the current context, due
to the simple fact that the conformal transformations involved do \emph{not}
preserve the minimal geodesics.

We need to point out, however, that in applying Colding-Naber's arugment to
prove Theorem~\ref{thm: main2}, the following difficulty has to be
overcomed: It is well-known that the comparison geometry of the Bakry-\'Emery
Ricci curvature depends, not only on the tensor lower bound, but also on the
gradient bound of the potential function. Yet basic properties of complete
Ricci shrinkers tell that the potential function has its gradient controlled,
in magnitude, by a linear function of the distance to the base point --- from
(\ref{E101}), (\ref{eqn: positive_scalar}) and Lemma \ref{L100}, we have 
\begin{align}\label{eqn: gradient_f}
\forall x\in M,\quad 2|\nabla f|(x)\ \le\ d(p,x)+\sqrt{2m}.
\end{align}
This growing gradient bound implies that the estimates one could obtain from the
Bakry-\'Emery Ricci curvature lower bound become worse as one moves further and
further away from the base point. 

While there have been extensive studies in the literature (see, for instance,
\cite{WW09}, \cite{WZ17} and \cite{ZZ17}) for manifolds with a uniform
Bakry-\'Emery Ricci curvature lower bound and a uniform bound on the gradient
of the potential function, they are not directly applicable to the case we are
handling: The assumption of a uniform bound on the gradient of the potential
functions is valid when studying the metric tangent cone of a fixed point in a
Ricci shrinker limit space (see \cite{WZ17} and Section 2.3), yet from
(\ref{eqn: gradient_f}) we see that in order to study the global properties of
the regular part in the limit space, such as the strong convexity we just
mentioned, it is necessary to develop estimates adapted to \emph{changing}
gradient control on the potential functions.

Therefore, we follow Colding-Naber's original idea, but we also have to start
from rebuilding the most basic estimates --- the situation is similar to the
recent joint work \cite{HW20a} by the first- and the third-named authors, where
Colding-Naber's estimates are generalized in another direction to control
the spreading of minimal geodesics emanating from a submanifold. 

Now let us recall that the major novelty in Colding-Naber's proof is a
parabolic smoothing of the distance function. This technique relies on a
uniform Ricci curvature lower bound and its resulting Li-Yau heat kernel
estimates~\cite{LY86}. On a Ricci shrinker, since the Ricci curvature lower
bound is replaced by a Bakry-\'Emery Ricci curvature bound, it is more natural
to consider, instead, the $f$-heat kernel, which is the heat kernel with
respect to the weighted measure $\mu_f$ whose density is defined as
$\dmu_f:=e^{-f}\dvol_g$. Roughly speaking, the corresponding $f$-heat kernel
bounds follow from a local volume doubling property and a local
$L^2$-Poincar\'e inequality of $\mu_f$.

In fact, we will give a new $f$-heat kernel estimate (Theorem \ref{T201}), and
its applications to the smoothing of distance functions, on more general
classes of manifolds, defined by only extracting the necessary analytic
properties:
\begin{defn}\label{defn:moduli}
Given a positive smooth nondecreasing function $F(r)$ on $[0,\infty)$, and a
number $K\ge 0$, $\mathcal N_m(F,K)$ is defined to be the class of pointed
smooth metric measure space $(M^m,p,g,\mu_f)$, where $\forall U\subset
M$ open, $\mu_f(U):=\int_Ue^{-f} \dvol_g$, such that
 \begin{itemize}
 \item[(a).] $(M^m,g)$ is a complete $m$-dimensional smooth Riemannian manifold.
 \item[(b).] There exists a $C^2$ function $f$ on $M$ such that
$f$ achieves a global minimum at some $p\in M$,
and $\forall x\in M$, $\max\{|f|(x),|\na f|^2(x)\} \le F^2(d_g(p,x))$.
 \item[(c).] The Bakry-\'Emery Ricci curvature satisfies
$Rc_f \coloneqq Rc+Hess_f\ \ge\ -Kg.$
\end{itemize}

In addition, the subclass $\mathcal N_m(F,K;V_0) \subset \mathcal
N_m(F,K)$ consists of all manifolds satisfying the
\begin{itemize}
 \item[(d).] Non-collapsing condition: 
$\mu_f(B(p,1))\ \ge\ V_0.$
\end{itemize}

We also define a subclass  $\mathcal{M}_m(F,K;V_0)\subset N_m(F,K;V_0)$ (see
\cite[Definition 10.1]{LLW17}) as all manifolds further satisfying the
following:
\begin{itemize}
 \item[(e).] $\forall x\in M$, $Rc_f(x)\le Kg(x)$ and $|R|(x)+|\nabla f|^2(x)
 \le F^2(d_g(p,x))$.
\end{itemize}
\end{defn}

Obviously, if we consider the above mentioned weighted measure, then by
(\ref{eqn: gradient_f}) any Ricci shrinker $(M^m,p,g,f)$ belongs to
$\mathcal N_m(F_{RS},1/2)$ for the linear function
$F_{RS}(t):=\frac{1}{2}(t+\sqrt{2m})$ (for any $t>0$). By \cite[Lemma
2.5]{LLW17}, we have for $c_m:=(4\pi)^{\frac{m}{2}}e^{-2^{4m+7}}$,
\begin{align*}
\mathcal{M}_m(A)\ \subset\
\mathcal{M}_m(F_{RS}, \frac{1}{2};c_me^{-A}).
\end{align*}
 For the sake of simplicity, we
will also denote any manifold $(M,p,g,\mu_f)\in \mathcal{N}_m(F,K)$ by
$(M,p,g,f)$.

\begin{remark}\label{rmk: RCD}
We point out any $(M^m,p,g,f) \in \mathcal N_m(F,K)$ is a smooth
$RCD(K,\infty)$ space, whose definition can be found in \cite{AGS14}, see also
\cite{LV09}, \cite{KT06a} and \cite{KT06b}. Theorem~\ref{thm: main2} shows that
the boundary of the moduli space $\mathcal{M}_m(A)$, in the
pointed-$\hat{C}^{\infty}$-Cheeger-Gromov topology, consists of conifold Ricci
shrinkers. By~\cite{GMS}, we see that these boundary points of
$\mathcal{M}_m(A)$ provide natural and novel examples of non-trivial
$RCD(\frac{1}{2},\infty)$ spaces.
\end{remark}

In fact, on a pointed-Gromov-Hausdorff limit of a sequence of manifolds in
$\mathcal{N}_m(F,K)$, one could define the regular part purely in terms of
metric tangent cones: By blowing-up the metric, the effect of the potential
function will be neglected, and the metric tangent cone could be defined by
Gromov's compactness theorem, see Section 2.3. Following a similar manner as
\cite{ChCo97}, the concept of regular-singular decomposition in
Theorem~\ref{thm: main2} could be generalized to the pointed-Gromov-Hausdorff
limits of sequences of manifolds in $\mathcal{N}_m(F,K)$. There is also, as
discussed in Section 2.2, a natural limit measure on a pointed-Gromov-Hausdorff
limit, making the convergence a pointed-measured-Gromov-Hausdorff convergence,
see Definition~\ref{def: pmGH}. In the last section, we will follow
Colding-Naber's idea to prove a H\"older continuity theorem (Theorem~\ref{thm:
Holder_continuity}) for manifolds in $\mathcal{N}_m(F,K)$, and supplement the
necessary details in extending limit minimal geodesics (Lemma~\ref{lem:
extension}) on the pointed-measured-Gromov-Hausdorff limits of manifolds in
$\mathcal{N}_m(F,K)$. Both of these results become indispensable ingredients in
proving the main theorem of the paper:
\begin{thm}
Let a sequence $\{(M_i,p_i,g_i,f_i)\}\subset \mathcal{N}_m(F,K)$ converge
to $(X,p_{\infty},d_{\infty},\nu_{\infty})$ in the
pointed-measured-Gromov-Hausdorff topology. Then there is a unique natural
number $k\le m$, such that $\nu_{\infty}(\mathcal{R}_k)>0$. Moreover,
$\mathcal{R}_k$ is both $\nu_{\infty}$-almost everywhere convex and weakly
convex in $X$.
\label{thm: main3}
\end{thm}

Here $\nu_{\infty}$ is a limit measure as just mentioned (see
Proposition~\ref{prop: renormalized_measure}), and we call a set $S\subset X$
to be $\nu_{\infty}$-\emph{almost everywhere convex} if for $\nu_{\infty}\times
\nu_{\infty}$ almost every pair of points $(x,y)\in S\times S$, there is a
minimal geodesic entirely contained in $S$ that connects them. Also we call a
set $S\subset X$ to be \emph{weakly convex} if
\begin{align*}
\forall x,y\in S,\quad d_X(x,y)\ =\ \inf_{\sigma\subset S}|\sigma|,
\end{align*}
where the infimum is taken over all curves $\sigma$ connecting $x$ and $y$, and
entirely contained in $S$.

\begin{remark}
The main theorem of Colding-Naber~\cite{CN12} can be regarded as a special case of Theorem~\ref{thm: main3} by letting $F \equiv 0$. 
Also, Theorem~\ref{thm: main2} is a natural consequence of Theorem~\ref{thm: main3}, after
invoking the regularity improvement made in~\cite{LLW17}. The proof of Theorem~\ref{thm: main3}, however, is independent of~\cite{LLW17}. Actually, Theorem~\ref{thm:
Holder_continuity}, as the main step in proving Theorem~\ref{thm: main3}, is by
itself a certain regularity improvement result --- it takes care of the lowest
level of regularity, while the major concern of \cite{LLW17} focuses on higher
regularities. Also compare Remark~\ref{rmk: LLW_necessity}.
\end{remark}

The paper is arranged as following: after recalling the necessary background in
Section 2, we will present the useful analytic properties about the manifolds in
$\mathcal{N}_m(F,K)$ in Section 3, and in Section 4 we finish the proof of the
main theorem, Theorem~\ref{thm: main3}, and its consequence Theorem~\ref{thm:
main2}. The following notations are employed throughout the paper:
\begin{enumerate}
  \item $D$ denotes a large positive constant, say $D>10m$;
  \item $Rc$, $R$ denote the Ricci and scalar curvature respectively;
  \item $B_d(x,r)$ denote the geodesic $r$-ball centered at $x$, with metric
  structure induced by $d$ (the dependence of $d$ is sometimes omitted when no
  confusion is caused);
  \item $A(x,r_1,r_2):=B(x,r_2)\backslash B(x,r_1)$ for $r_2\ge r_1>0$;
  \item $\mathcal{R}_k$ denotes the $k$-stratum of the regular part,
  $\mathcal{R}$ the entire regular part, and $\mathcal{S}$ the singular part.
  
\end{enumerate}

\section{Background}

In this section we recall the basic analytic properties of Ricci shrinkers and
manifolds in the moduli $\mathcal{N}_m(F,K)$, and discuss various concepts
related to the regular-singular decomposition of the pointed-Gromov-Hausdorff
limits.
\subsection{Basic estimates for Ricci shrinkers and manifolds in $\mathcal
N_m(F,K)$} 
It is immediate from the definition that the function $f$ must be smooth.
Moreover, we have the following point wise estimate of $f$ by \cite[Lemma
$2.1$]{HM11}:

\begin{lem}
\label{L100}
Let $(M^m,g,f)$ be a Ricci shrinker. Then there exists a point $p \in
M$ where $f$ attains its infimum and $f$ satisfies the quadratic growth estimate
\begin{align*} 
\frac{1}{4}\left(d(x,p)-5m \right)^2_+ \le f(x) \le \frac{1}{4}
\left(d(x,p)+\sqrt{2m} \right)^2
\end{align*}
for all $x\in M$, where $a_+ :=\max\{0,a\}$. Moreover, if $p_1,p_2\in M$ are two
distinct minima of $f$, then $d(p_1,p_2)\le \sqrt{2m}+5m$.
\end{lem}

In other words, $f$ increases like a quadratic function. Moreover, it follows
from \eqref{E101} that $f$ is nonnegative and $|\na f|$ increases at most
linearly. From now on, whenever we talk about a pointed Ricci shrinker, we fix
one of the minima of $f$ as the base point. Recall that associated to the Ricci
shrinker metric structure, there is a natural measure $\mu_f$, with
density $\dmu_f = e^{-f}\dvol_g$. It is clear from Lemma \ref{L100} and
the next lemma that $\mu_f$ is a finite measure, see \cite[Corollary
1.1]{CZ10}.
\begin{lem}[Lemma 2.2 of~\cite{HM11} and Theorem 1.2 of~\cite{CZ10}]
\label{L100a}
For each dimension $m$, and for any non-compact Ricci shrinker
$(M^m,p,g,f)$ with $p \in M$ being a minimal point of $f$,
\begin{align}
\frac{|B(p,D)|}{|B(p,D_0(m))|}\ \le C_1(m)\ D^m,
\end{align}
for any $D\ge D_0(m)$, with $D_0(m):=\sqrt{2m+4}+5m$ and
$C_1(m):=2\lc\frac{2}{m+2}\rc^{m/2}$.
\end{lem}

For the perspective of our future discussion in this paper, we point
out that properties described in Lemma 2.1 and Lemma 2.2 are, besides an upper
bound of the Bakry-\'Emery Ricci tensor, the only special ones that Ricci
shrinkers enjoy over general manifolds in the moduli
$\mathcal{N}_m(F_{RS},\frac{1}{2})$.

The fundamental property that general manifolds in $\mathcal N_m(F,K)$ enjoy, is
a uniform lower bound of the Bakry-\'Emery Ricci tensor. The basic comparison
geometry in this setting is discussed in the work of Wei-Wylie~\cite{WW09}.
Notice that a new phenomenon for such comparison is its dependence on the
gradient bound of the potential function. For manifolds in $\mathcal{N}_m(F,K)$,
we therefore need to adjust the relevant properties to incorporate the changing
gradient bound controlled by $F$.

We will let $M_K^m$ denote the usual $m$-dimensional space
form of constant sectional curvature $-K$. Picking a base point $o\in M_K^m$, we
let $\mathcal{A}_K^{m-1}(r)$ denote the area form of the geodesic $r$-sphere
around $o$, for some $r>0$. 

Now we state the following upper bound of the $f$-Laplacian
$\Delta_f:=\Delta-\la \na f, \na \cdot \ra$ in terms of the distance function,
see \cite[Theorem $1.1$]{WW09}.
\begin{lem}
\label{T102}
Let $(M,p,g,f)\in \mathcal N_m(F,K)$. Then
\begin{align*} 
\forall q\in B(p,D),\quad \Delta_f r\ \le\ \frac{m-1}{r}+(m-1)\sqrt{K}+F(2D)
\end{align*}
on $B(q,D)$, where $r(\cdot):=d(q,\cdot)$.
\end{lem}
\begin{proof}
It follows from \cite[Theorem $1.1$,(a)]{WW09} that
\begin{align*} 
\Delta_f r\ \le\ m_K(r)+F(2D),
\end{align*}
where $m_K(r)$ is the mean curvature of the geodesic sphere in $M_K^m$. Since
the mean curvature is bounded as
\begin{align*} 
m_K(r)\ \le\ (m-1)\sqrt{K} \coth{\sqrt{K}r}\ \le\ 
\frac{m-1}{r}+(m-1)\sqrt{K},
\end{align*}
where we have used the elementary inequality $x \coth x \le 1+x$ for $x \ge 0$.
\end{proof}

\begin{cor}
\label{C102}
Let $(M,p,g,f)\in \mathcal N_m(F,K)$ and fix $q\in B(p,D)$. If we denote the
weighted volume form on sphere with respect to $q$ by $\mathcal
A_f(r,\mathbf{v})$ with $\mathbf{v}\in S_qM$ (the unit tangent vectors at $q\in
M$), then
\begin{align*} 
\frac{\mathcal A_f(r,\mathbf{v})}{\mathcal A_f(s,\mathbf{v})}\ \le\
\frac{e^{F(2D)r}\mathcal{A}^{m-1}_K(r)}{e^{F(2D)s}\mathcal{A}^{m-1}_K(s)}\ \le\
\left(\frac{r}{s}\right)^{m-1}e^{\lc (m-1)\sqrt{K}+F(2D)\rc(r-s)}.
\end{align*}
\end{cor}

\begin{proof}
It is clear from the definition $\mathcal A_f(r,\mathbf{v})=e^{-f}
\mathcal{A}(r,\mathbf{v})$ that
\begin{align*} 
\left(\ln \mathcal A_f(t,\mathbf{v})\right)'\ &=\
m_f(t,\mathbf{v})\ =\ \Delta_f r(t,\mathbf{v})\\
 &\le\ m_K(t)+F(2D)\ =\
 \left( \ln \mathcal{A}^{m-1}_K(t)\right)'+F(2D)\\
 &\le\ (m-1)r^{-1}+(m-1)\sqrt{K}+F(2D).
\end{align*}
Then integrating from $0<s$ to $r \le D$, we have the desired estimates.
\end{proof}
By \cite[Lemma 3.2]{Shunhui}, we could further estimate, for $0<r_1\le
r_2$, $0<s_1\le s_2$, $s_1<r_1$ and $s_2<r_2$, that $\forall \mathbf{v} \in
S_qM$ (the unit sphere in the tangent space of $q\in M$),
\begin{align}\label{eqn: area_comparison_translated}
\frac{\int_{s_2}^{r_2}\mathcal{A}_f(t,\mathbf{v})\
\text{d}t}{\int_{s_1}^{r_1}\mathcal{A}_f(t,\mathbf{v})\ \text{d}t}\ \le\ 
\frac{\int_{s_2}^{r_2}e^{F(2D)t}\mathcal{A}_K^{m-1}(t,)\
\text{d}t}{\int_{s_1}^{r_1}e^{F(2D)t}\mathcal{A}_K^{m-1}(t,)\
\text{d}t},
\end{align}
and by the mean value theorem,
\begin{align*}
\frac{\mathcal{A}_f(r_2,\mathbf{v})}{\int_{s_1}^{r_1}
\mathcal{A}_f(t,\mathbf{v})\ \text{d}t}\ \le\
\frac{e^{F(2D)r_2}\mathcal{A}_f(r_2,\mathbf{v})}{\int_{s_1}^{r_1}
e^{F(2)t}\mathcal{A}_f(t,\mathbf{v})\ \text{d}t}.
\end{align*} 
Integrating the above inequalities in all tangent directions and arguing as
\cite[Theorem 3.1]{Shunhui}, we have the following volume comparison theorem
(compare also \cite[Theorem $1.2$]{WW09}):
\begin{thm}
\label{T101}
  Let $(M^m,p, g,f)\in \mathcal{N}_m(F,K)$, and fix $q \in M$ such that
  $B(q,r_2)\subset B(p,D)$. Then for $s_2\ge s_1> 0$ and $r_1\in (0,r_2)$ such
  that $s_2\le r_2$ and $s_1\le r_1$, we have the following estimates:
     \begin{align*}
       \frac{\mu_f(A(q,r_2,r_1))}{\mu_f(A(q,s_2,s_1))}\ &\leq\ e^{F(D)(r_2-s_1)}
       \frac{Vol_{K}^{m}(r_2)-Vol_K^m(s_2)}{Vol_{K}^{m}(r_1)-Vol_K^m(s_1)},\\
       \frac{\mu_f(B(q,r_2))}{\mu_f(B(q,s_2))}\ &\le\
       e^{F(D)r_2}\frac{Vol_K^m(r_2)}{Vol_K^m(s_2)},\\
       \text{and}\quad \frac{\mu_f(\partial B(q,r_2))}{\mu_f(B(q,s_2))}\ &\le\
       e^{F(D)r_2}\frac{Area_K^{m-1}(r_2)}{Vol_K^m(s_2)},
     \end{align*}
     where for any $r>0$, $Vol_{H}^{m}(r)$ is the volume of the geodesic
     $r$-ball in $M_{K}^{m}$, and $Area_K^{m-1}(r)$ is the area of its boundary.
\end{thm}

Especially, when $r_2=2s_2$ in the second inequality above, we have the
following volume doubling property: when $B(q,2r)\subset B(p,D)$,
\begin{align}
\label{eqn: volume_doubling}
\mu_f(B(q,2r))\ \le\ C_2(m,F,K,D)\mu_f(B(q,r)),
\end{align}
where the doubling constant $C_2(m,F,K,D):=2^{m-1}e^{D((m-1)\sqrt{K}+F(D))}$ is
uniform for $D$.

For later applications, we also need the following segment inequality,
originally due to Cheeger-Colding \cite{ChCo96} for manifolds with uniform
Ricci lower bound.

\begin{thm}[Segment inequality]
\label{T104}
Let $(M,p,g,f)\in \mathcal N_m(F,K)$. For any $D>0$, there
exists a constant $C_{Seg}=C_{Seg}(m,F,K,D)$ such that if $U$ is a geodesically
convex set in $B(q,D)\subset B(p,2D)$, then for any set $V \subset U$,
\begin{align*} 
\int_{V\times V} \mathcal F_u(x,y)\ \dmu_f(x)\dmu_f(y) \le
C_{Seg}\mu_f(V)(\text{diam}\ U) \int_U u\ \dmu_f
\end{align*}
where $u$ is a nonnegative continuous function on $U$ and 
\begin{align*} 
\mathcal F_u(x,y) \coloneqq \inf_{\gamma} \int_0^{d(x,y)}u(\gamma (t))\ dt,
\end{align*}
with infimum being taken over all minimal geodesics connecting $x$ and $y$.
\end{thm}
\begin{proof}
It follows from Corollary \ref{C102} and an argument
based on the original one of Cheeger-Colding's. For the sake of completeness we
write down the technical details, see also \cite{Huang17} for a version only for
Ricci shrinkers.

We may consider $\mathcal{F}_u(x,y)=\mathcal{F}_u^+(x,y)+\mathcal{F}_u^-(x,y)$
where 
\begin{align*}
\mathcal{F}_u^+(x,y)\ 
 :=\ \inf_{\{\gamma_{xy}\}}\int_{\frac{d(x,y)}{2}}^{d(x,y)}u(\gamma_{xy}(t))\
\text{d}t\quad \text{and}\quad
\mathcal{F}_u^-(x,y)\
:=\ \inf_{\{\gamma_{xy}\}}\int_0^{\frac{d(x,y)}{2}}u(\gamma_{xy}(t))\
\text{d}t.
\end{align*}
 Since $\mathcal{F}^+_u(x,y)=\mathcal{F}^-_u(y,x)$, by Fubini's
theorem, $$\int_{V\times V}\mathcal{F}_u^+(x,y)\ \dmu_f(x)\dmu_f(y)\ =\
\int_{V\times V}\mathcal{F}_u^-(x,y)\ \dmu_f(x)\dmu_f(y),$$ and so we only need
to do the estimate for $\mathcal{F}_u^+$. For any $x\in V$ and any
$\mathbf{v}\in S_xM$ fixed, define $d_{x,\mathbf{v}}:=\min\{t>0:
\exp_x(t\mathbf{v})\in \partial U\}$, also denote
$\gamma_{\mathbf{v}}(t)=\exp_x(t\mathbf{v})$. Then $\forall t\in
(0,d_{x,\mathbf{v}})$, by Corollary \ref{C102},
\begin{align*}
\mathcal{F}_u^+(\gamma_{\mathbf{v}}(t\slash
2),\gamma_{\mathbf{v}}(t))\ \dmu_f(\gamma_{\mathbf{v}}(t))\ 
&\le\ \left(\int_{\frac{t}{2}}^tu(\gamma_{\mathbf{v}}(s))\ \text{d}s\right)\
\area_f(\mathbf{v},t)\ \text{d}t\\
&\le\  C\left(\int_{\frac{t}{2}}^tu(\gamma_{\mathbf{v}}(s))\
\area_f(\mathbf{v},s)\text{d}s\right)\ \text{d}t.
\end{align*}
By the assumption on $V\subset U$, for almost every $y\in X$, there exists some
$\mathbf{v}\in S_xM$ such that $\gamma_{\mathbf{v}}(d(x,y))=y$, we have
\begin{align*}
\int_{V}\mathcal{F}_u^+(x,y)\ \dmu_f(y)\ &\le\
\int_{S_xM}\int_0^{d_{x,\mathbf{v}}}
\mathcal{F}_u^+(\gamma_{\mathbf{v}}(t\slash
2),\gamma_{\mathbf{v}}(t))\ \area_f(\mathbf{v},t)\ \text{d}t\text{d}\mathbf{v}\\
&\le C\diam U
\int_{S_xM}\int_0^{d_{x,\mathbf{v}}}u(\gamma_{\mathbf{v}}(s))\
\area_f(\mathbf{v},s)\ \text{d}s\text{d}\mathbf{v}\\
&\le C\diam U \int_{U}u\ \dmu_f.
\end{align*}
Finally, integrating the above inequality for $x\in X$, we get
\begin{align*}
\int_{V}\int_{V}\mathcal{F}_u^+(x,y)\ \dmu_f(y)\dmu_f(x)\ &\le\
C\mu_f(V)\diam U \int_U u\ \dmu_f.
\end{align*}
\end{proof}

With the help of the volume doubling property and the segment inequality, the
following local $L^2$-Poincar\'e inequality holds, see \cite{ChCo00} for a proof.

\begin{prop}[Local $L^2$-Poincar\'e inequality]\label{T:poincare}
Let $(M,p,g,f)\in \mathcal N_m(F,K)$. For any $D>0$, there
exists a constant $C_P=C_P(m,F,K,D)$ such that for any $B(q,r) \subset
B(p,D)$,
\begin{align*} 
\aint_{B(q,r)}\left|u-\aint_{B(q,r)}u\,d\mu_f\right|^2\,\dmu_f\ \le\
C_Pr^2\aint_{B(q,r)}|\na u|^2\,d\mu_f
\end{align*}
for any $u \in C^1(B(q,r))$.
\end{prop}

\begin{remark}
Throughout this paper we will let $\aint$ denote the average over a set whose
total mass is weighted against the measure in the integral, that is to say, for
any integrable function $u$ on $B(q,r)$, 
\begin{align*}
\aint_{B(q,r)}u\ \dmu_f\ :=\ \frac{1}{\mu_f(B(q,r))}\int_{B(q,r)}u\ \dmu_f.
\end{align*}
\end{remark}

Moreover, the local volume doubling with the local $L^2$-Poincar\'e inequality
will imply the following local Sobolev inequality, see \cite{LSC92}.

\begin{prop}[Local $L^2$-Sobolev inequality]\label{T:Sobolev}
Let $(M,p,g,f)\in \mathcal N_m(F,K)$. For any $D>0$, there
exists a constant $C_{Sob}=C_{Sob}(m,F,K,D)$ such that for any $B(q,r)
\subset B(p,D)$,
\begin{align*} 
\lc\int_{B(q,r)}u^{\frac{2m}{m-2}}\,\dmu_f \rc^{\frac{m-2}{m}}\ \le\ 
\frac{C_{Sob}r^2}{\mu_f(B(q,r))^{\frac{2}{m}}}\int_{B(q,r)}|\na
u|^2+r^{-2}u^2\,d\mu_f
\end{align*}
for any $u \in C_c^1(B(q,r))$.
\end{prop}

\subsection{Weak-compactification of the moduli spaces}
To begin this section, we first present the following weak-compactness theorem
of $\mathcal N_m(F,K)$.

\begin{thm}[\textbf{Weak-compactness of $\mathcal N_m(F,K)$}] \label{T:weak}
Let $\{(M_i^m,p_i, g_i, f_i)\}$ be a sequence in $\mathcal N_m(F,K)$, and let
$d_i$ denote the length structure induced by $g_i$. By passing to a subsequence
if necessary,  we have
\begin{align}
   (M_i^n, p_i,d_i, f_i)  \longright{pointed-Gromov-Hausdorff}
   \left(X,p_{\infty},d_{\infty},f_{\infty}\right),
\end{align}
where $(X,d_{\infty})$ is a length space, $f_{\infty}$ is a Lipschitz function
on $X$.
\end{thm}

\begin{proof}
For any $D>0$, we have a uniform volume doubling constant on $B(p_i,D)\subset
M_i$ by (\ref{eqn: volume_doubling}). Then it follows from the standard ball
packing argument of Gromov, see \cite[Proposition $5.2$]{Gromov} that
\begin{align*}
   (M_i^m, p_i,d_i)  \longright{pointed-Gromov-Hausdorff}
   \left(X,p_{\infty},d_{\infty} \right).
\end{align*}
In addition, since $|\na f_i| \le F(D)$ on $B(p_i,D)\subset M_i$, it is clear
from the Arzela-Ascoli theorem that $f_i$ converge to a locally Lipschitz limit
function $f_{\infty}$. Moreover, $\|f_{\infty}\|_{\text{Lip}} \le F(D)$ on
$B(p_{\infty},D)\subset X$.
\end{proof}

Besides the pure metric structure, we also have a limit measure on the
pointed-Gromov-Hausdorff limit, and we define the
pointed-measured-Gromov-Hausdorff convergence as following:
\begin{defn}\label{def: pmGH}
Let $\{(M_i,p_i,d_i,\mu_i)\}$ be a sequence of metric measure spaces, we say
this sequence pointed-measured-Gromov-Hausdorff converges to a metric measure
space $(X,p,d,\mu)$ if there exist a sequence of radii $D_k\uparrow \infty$, and
pointed-Gromov-Hausdorff approximations $\Phi_{ik}:B_{d_i}(p_i,D_k)\to
B_{d}(p,D)\subset X$, such that $(\Phi_{ik})_{\ast}\mu_i\to \mu$ in
$C_0(B_d(p,D_k))^{\ast}$, the dual space of all continuous functions on
$B_d(p,D_k)$, vanishing on $\partial B_d(p,D_k)$. 
\end{defn} 
Now for each $(M,p,g,f)\in \mathcal{N}_m(F,K)$ we define the renormalized
measure $\nu_f:=\mu_f(B(p,1))^{-1}\mu_f$, and have the
following proposition in analogy to the case of manifolds with a uniform Ricci
curvature lower bound~\cite{ChCo97}:
\begin{prop}\label{prop: renormalized_measure}
Consider a sequence $\{(M_i,p_i,g_i,f_i)\}\subset \mathcal{N}_m(F,K)$
that converges to a pointed metric space $(X,p_{\infty},d_{\infty},f_{\infty})$
in the pointed-Gromov-Hausdorff topology. Then there is a subsequence, denoted
by $\{(M_{i_j},p_{i_j},g_{i_j},f_{i_j})\}$, and a Randon measure
$\nu_{\infty}$ such that $\{(M_{i_j},p_{i_j},g_{i_j},\nu_{f_{i_j}})\}$
converges to $(X,p_{\infty},d_{\infty},\nu_{\infty})\}$ in the
pointed-measured-Gromov-Hausdorff topology. Moreover, $\nu_{\infty}$ satisfies
the following conditions:
\begin{enumerate}
  \item[(1)] $\forall x\in X$ and $\forall r>0$, suppose $M_{i_j}\ni
  x_{i_j}\longright{GH}x$, then
  \begin{align*}
  \nu_{\infty}(B(x,r))\ =\ \lim_{j\to \infty}\nu_{f_{i_j}}(B(x_{i_j},r));
  \end{align*}
  \item[(2)] $\forall x\in X$ and $\forall r_2\ge r_1>0$,
  \begin{align*}
       \frac{\nu_{\infty}(B(x,r_2))}{\nu_{\infty}(B(x,r_1))}\ \leq\
       e^{C_F(x,r_2)}
       \frac{Vol_{K}^{m}(r_2)}{Vol_{K}^{m}(r_1)},
     \end{align*}
  where
  $C_F(x,r_2):=F(r_2+d_{\infty}(x,p_{\infty}))(r_2+d_{\infty}(x,p_{\infty}))$.
\end{enumerate}
Furthermore, any Randon measure on $X$ satisfying the above two conditions
agrees with $\nu_{\infty}$.
\end{prop}
\noindent In the case of Gromov-Hausdorff convergence of Riemannian manifolds
with a uniform Ricci curvature lower bound, this was shown in~\cite[Section
1]{ChCo97}, where the natural measure was renormalized by the volume of a unit
ball (centered at some base point chosen on the manifold), and the renormalized
measured was shown to converge, uniformly on compact subsets, to a limit Randon
measure on the pointed-Gromov-Hausdorff limit. By Theorem~\ref{T101}, we could
easily see the following estimates: $\forall x,y\in M$, $\forall r\ge s>0$ such that
$B(x,r)\subset B(p,D)$ and $B(y,s)\subset B(p,D)$, 
\begin{align*}
\frac{\mu_f(B(x,r))}{\mu_f(B(y,s))}\ &\le\
e^{2DF(D)}\frac{Vol^m_K(r+d(x,y))}{Vol_K^m(s)},\\
\text{and}\quad \frac{\mu_f(B(x,r))}{\mu_f(B(y,s))}\ &\ge\ \begin{cases}
e^{-2DF(D)}\frac{Vol_K^m(r)}{Vol_K^m(s+d(x,y))}\quad &\text{when}\ r\le
s+d(x,y),\\
1\quad &\text{when}\ r\ge s+d(x,y).
\end{cases}
\end{align*}
Obviously, these estimates have~\cite[Estimates (1.2)-(1.4)]{ChCo97} as their
counterparts for manifolds with a uniform Ricci curvature lower bound,
and consequently, the same constructions as~\cite[Theorem 1.6, Theorem
1.10]{ChCo97} work in our situation to deduce the last proposition.

If we further focus on the sub-collection of all $m$-dimensional non-compact
Ricci shrinkers, the natural measure $\mu_f$, as pointed out
by~\cite[Corollary 1.1]{CZ10}, has finite total mass. Due to the growth
property of $f$ (Lemma~\ref{L100}), $\mu_f$ has an essentially
canonical choice of base point --- one of the minima of $f$. We will therefore
consider a Ricci shrinker $(M,p,g,f)$ together with the canonical probability
measure $\rho:=\mu_f(M)^{-1}\mu_f$. 
In fact we have the following:
\begin{prop}
Let $\{(M_i,p_i,g_i,f_i)\}$ be a sequence of $m$-dimensional non-compact Ricci
shrinkers that pointed Gromov-Hausdorff converges to a metric space
$(X,p_{\infty},d_{\infty},f_{\infty})$, then there is a subsequence and a unique
Randon measure $\rho_{\infty}$ satisfying conditions (1) and (2) in
Proposition~\ref{prop: renormalized_measure}. Moreover, $\rho_{\infty}$ is a
probability measure.
\end{prop}
\begin{proof}
It only remains to show that the limit measure $\rho_{\infty}$ is a probability
measure. To see this we turn to the estimates in Lemma~\ref{L100} and
Lemma~\ref{L100a}.
For each $k\in \mathbb{N}$, define $D_k:=2^kD_0(m)$, and for the sake of
simplicity let $B_{i,k}$ denote $B(p_i,D_k)\subset M_i$. Then we could
estimate for each $i$ and each $k\ge
K_m:=\lceil\frac{1}{2}\log_2(18m)\rceil$:
\begin{align}\label{eqn: tail}
\begin{split}
\mu_{f_i}(M_i\backslash B_{i,k})\ =\
&\sum_{j=k}^{\infty}\int_{B_{i,j+1}\backslash B_{i,j}}e^{-f_i}\ \dvol_{g_i}\\
\le\
&\sum_{j=k}^{\infty}|B_{i,0}|D_{j+1}^me^{-\frac{1}{4}(D_{j}-5m)_+^2}\\
\le\ &|B_{i,0}|(2D_0(m))^m\int_k^{\infty}e^{-u}\ \text{d}u\\
=\ &e^{-k}(2D_0(m))^m|B_{i,0}|.
\end{split}
\end{align}
This estimate, combined with the inequality
\begin{align*}
\mu_{f_i}(B_{i,0})\ \ge\
e^{-\frac{1}{4}(D_0(m)+\sqrt{2m})^2}|B_{i,0}|,
\end{align*}
implies
\begin{align*}
\frac{\mu_{f_i}(M_i\backslash B_{i,k})}{\mu_{f_i}(M_i)}\ \le\
e^{-k}(2D_0(m))^me^{\frac{1}{4}(D_0(m)+\sqrt{2m})^2}.
\end{align*} 
Therefore, for each $i$ and each $k\ge K_m$, 
\begin{align}
 1-e^{-k}(2D_0(m))^me^{\frac{1}{4}(D_0(m)+\sqrt{2m})^2}\ \le\
 \rho_i(B_{i,k})\ \le\ 1.
\end{align}

Taking the pointed-Gromov-Hausdorff convergence with the radii
$\{D_k\}_{k\ge K_m}$, it is easy to see that the limit measure $\rho_{\infty}$
has unit total mass, whence a probability measure. 
\end{proof}

If we only consider the class $\mathcal M_m(F,K;V_0)$, it follows from
\cite[Theorem $10.1$]{LLW17} that

\begin{thm} \label{thmin:a}
Consider a sequence $\{(M_i^m,p_i, g_i,f_i,)\}\subset \mathcal M_m(F,K;V_0)$
such that
\begin{align*}
   (M_i^m,p_i, d_i, f_i)  \longright{pointed-Gromov-Hausdorff}
   \left(M_{\infty},p_{\infty}, d_{\infty},f_{\infty}\right).
\end{align*} 
Then $M_{\infty}$ has a regular-singular decomposition $M_{\infty}=\mathcal{R}
\cup \mathcal{S}$ with the following properties.
 \begin{itemize}
 \item[(a).] The singular part $\mathcal{S}$ is a closed set of Minkowski
 codimension at least $4$.
 \item[(b).] The regular part $\mathcal{R}$ is an $m$-dimensional open manifold
 with a $C^{1,\alpha}$ metric $g_{\infty}$ and $f_{\infty}$ is a $C^{1,\alpha}$
 function on $\mathcal R$.
 \end{itemize}
The convergence can be improved to 
\begin{align}
    (M_i, p_i, g_i,f_i)  \longright{pointed-\hat{C}^{1,\alpha}-Cheeger-Gromov}
    \left(M_{\infty}, p_{\infty}, g_{\infty}, f_{\infty} \right),     
    \label{eqn:CA25_10}
\end{align}
and the metric structure induced by smooth curves in $(\mathcal{R},
g_{\infty})$ coincides with $d_{\infty}$.

Moreover, the limit renormalized measure $\nu_{\infty}$ on
$M_{\infty}$ is defined as following: $\nu_{\infty}$ vanishes on
$\mathcal{S}$, and on $\mathcal{R}$ it has density
$V_{\infty}^{-1}\dmu_{f_{\infty}}$, where
 \begin{align*}
 V_{\infty}\ :=\ \lim_{i\to \infty}\mu_{f_i}(B(p_i,1)).
 \end{align*}
\end{thm}

\begin{remark}\label{rmk: LLW_necessity}
In fact, combining the work of Wang-Zhu~\cite{WZ17} and
Zhang-Zhu~\cite{ZZ17}, the pointed-Gromov-Hausdorff convergence could
already be improved to the pointed-$\hat{C}^{\alpha}$-Cheeger-Gromov
convergence. However, without the endeavors made in~\cite{LLW17}, one cannot
directly improve the regularity to pointed-$\hat{C}^{1,\alpha}$-Cheeger-Gromov
convergence, let along the pointed-$\hat{C}^{\infty}$-Cheeger-Gromov convergence
for Ricci shrinkers in Theorem~\ref{thm: main2}.
\end{remark}

Note that the limit measure identities
\begin{align*}
\nu_{\infty}\ =\ \begin{cases}
V_{\infty}^{-1}\mu_{f_{\infty}}\quad &\text{on}\quad \mathcal{R},\\
0\quad &\text{on}\quad \mathcal{S},
\end{cases}
\end{align*}
amounts to say that in the sub-collection $\mathcal{M}_m(F,K;V_0)$,
pointed-Gromov-Hausdorff topology is equivalent to
pointed-measured-Hausdorff topology. Therefore we will only need to
discuss the pointed-Gromov-Hausdorff topology on $\mathcal{M}_m(F,K;V_0)$. 

If the sequence in consideration actually consists of complete Ricci shrinkers,
we could promote the convergence to pointed-$\hat{C}^{\infty}$-Cheeger-Gromov
convergence, by the usual elliptic bootstrapping argument (see \cite{HM11}
and~\cite{LLW17}).
Also, the limit measure could be shown to be a probability measure
$\rho_{\infty}$ such that with $U_{\infty}:=\lim_{i\to \infty}\mu_{f_i}(M_i)$,
we have 
\begin{align*}
\rho_{\infty}\ =\ \begin{cases}
U_{\infty}^{-1}\mu_{f_{\infty}}\quad &\text{on}\quad \mathcal{R},\\
0\quad &\text{on}\quad \mathcal{S}.
\end{cases}
\end{align*}

\subsection{Regular-singular decomposition of the Gromov-Hausdorff limits}
The regular-singular decomposition of the pointed-Gromov-Hausdorff limit in
Theorem~\ref{thmin:a} could be discussed in the more general setting for
manifolds in $\mathcal{N}_m(F,K)$.

The definition of the regular part in the pointed-Gromov-Hausdorff limit
$(X,p_{\infty},d_{\infty},f_{\infty})$ of a sequence in $\mathcal{N}_m(F,K)$ is
based on the concept of metric tangent cones, as done in the case of manifolds
with a uniform Ricci curvature lower bound, see~\cite[Section 0]{ChCo97}.

To see the existence of a metric tangent cone for any point $x\in X$, we fix any
sequence of scales $r_j\to 0$ (assuming $r_j\in (0,1)$), and assume that
$M_{i}\ni x_{i}\longright{GH}x\in X$. We could then
consider the sequence of pointed metric spaces $\{(X,x,\tilde{d}_j)\}$, with the
rescaled metrics $\tilde{d}_j:=r_j^{-1}d_{\infty}$. Clearly, we have
$(M_i,p_i,r_j^{-2}g_i)\in \mathcal{N}_m(\tilde{F}_j,r_jK)\subset
\mathcal{N}_m(F,K)$, where$\tilde{F}_j:=r_jF$. Now for any $D>0$ and any $j$
fixed, we have
\begin{align*}
B_{r_j^{-2}g_{i}}(x_{i},D)\ \longright{Gromov-Hausdorff}\
B_{\tilde{d}_j}(x,D).
\end{align*}
Therefore regarding $B_{r_j^{-2}g_{i}}(x_{i},D)\subset
(M_{i},p_{i},r_j^{-2}g_i)$, we have, by Proposition~\ref{prop:
renormalized_measure}, that there is a limit renormalized measure
$\tilde{\nu}_j$ such that $\forall B_{\tilde{d}_j}(x',2r)\subset
B_{\tilde{d}_j}(x,D)$,
\begin{align*}
 \tilde{\nu}_j(B_{\tilde{d}_j}(x',2r))\ \le\
C_j(x,r,D)\tilde{\nu}_j(B_{\tilde{d}_j}(x',r)),
\end{align*}
where the sequence 
\begin{align*}
C_j(x,r,D)\ :=\ e^{F(d_{\infty}(x,p_{\infty})+D)D}Vol_{r_j^2K\slash
2}^m(2r)\slash Vol_{r_j^2K\slash 2}^m(r)
\end{align*} 
is uniformly bounded since $r_j\to 0$. Therefore the sequence of pointed metric
spaces $\{(X,x,\tilde{d}_j)\}$, equipped with measures $\tilde{\nu}_j$, have a
uniform doubling constant within the fixed distance $D$ to the base point. This
implies that the maximal number $N_j(x,r,D)$ of disjoint $r$-balls fitting into
$B_{\tilde{d}_j}(x,D)\subset (X,\tilde{d}_j)$ is uniformly bounded in $j\in
\mathbb{N}$, and Gromov's compactness theorem~\cite[Proposition $5.2$]{Gromov}
guarantees the existence of a complete metric space to which a subsequence in
$\{(X,x,\tilde{d}_j)\}$ converges in the pointed-Gromov-Hausdorff topology. This
limit metric space defines a \emph{metric tangent cone} of $X$ at $x$.

Now for each $k=1,2,\cdots,m$, we define, following~\cite[Definition
0.1]{ChCo97}, the $k$-regular part of $X$:
\begin{align}
\mathcal{R}_k\ :=\ \{x\in X:\ \text{any metric tangent cone at}\ x\
\text{is isometric to the Euclidean}\ k\text{-space}\}.
\label{eqn:PI10_1}
\end{align}
\noindent We also call $\cup_{k=1}^m\mathcal{R}_k$ the regular part of $X$,
denoted by $\mathcal{R}$, and $\mathcal{S}:=X\backslash \mathcal{R}$.

To justify the notation, we have the following characterization of the regular
part in the non-collapsing case:
\begin{thm} \label{T:noncollapsing}
Let $\{(M_i^m, p_i,g_i, f_i)\}$ be a sequence of manifolds in
$\mathcal{N}_m(F,K;V_0)$ such that
\begin{align*}
   (M_i^m, p_i,d_i, f_i)  \longright{pointed-Gromov-Hausdorff}
   \left(M_{\infty},p_{\infty},  d_{\infty},f_{\infty}\right).
\end{align*} 
Then $y \in \mathcal R$ if and only if
there exists a tangent cone at $y$ which is isometric to
$(\mathbb{R}^m,g_{Euc})$.
\end{thm}

\begin{proof}
We first prove that if $y \in B(p_{\infty},D)$ is a regular point, then any
tangent cone at $y$ is $(\mathbb{R}^m,g_{Euc})$. Otherwise, all tangent cones at
$y$ are isometric to $(\mathbb{R}^l,g_{Euc})$ for some integer $l<m$. Then for
any $\ep>0$ there exists a small $r=r(\ep)>0$ such that
  \begin{align*}
    d_{GH}\left\{ \left( B_{r^{-1}d_{\infty}}(y, 2), r^{-1} d_{\infty} \right),
    \left( B_{d_{Euc}}(0, 2),  d_{Euc} \right) \right\}<\ep.
  \end{align*}
Now for $i$ large enough such that
  \begin{align} \label{E701}
    d_{GH}\left\{ \left( B_{r^{-1}d_i}(y_{i}, 2), r^{-1}d_i \right), \left(
    B_{d_{Euc}}(0, 2),  d_{Euc} \right) \right\} <\ep
 \end{align}
where $y_i \to y$. We fix $k=10/\ep$ and consider a family of disjoint balls
$\{B(x_k,k^{-1}),\, k=1,2,\cdots,N_k\}$ such that $\{B(x_k,2k^{-1})\}$ cover
$B(0,1) \subset \mathbb{R}^l$. It is clear that $N_k \le k^l$. If we take
$x_{k,i} \to x_k$, then it is clear from \eqref{E701} that
$B_{r^{-1}d_i}(y_i,1)$ is covered by $\{B_{r^{-1}d_i}(x_{k,i},3k^{-1})\}$ if
$i$ is sufficiently large.

We next estimate the volume of $B_{r^{-1}d_i}(x_{k,i},3k^{-1})$ by using
Theorem \ref{T101}. Let $\bar f_i=f_i-f_i(x_{k,i})$ and $\bar g_i=r^{-2}g_i$,
then $\overline{Rc_{\bar f_i}}=Rc_{f_i}\ge r^2\bar g_i$. In addition,
$|\na_{\bar g_i}\bar f_i|=r|\na_{g_i}f_i| \le rF(D)$. It is clear from Theorem
\ref{T101} that
  \begin{align*}
|B_{r^{-1}d_i}(x_{k,i},3k^{-1})|_{\bar g_i} \le Ck^{-m}
  \end{align*}
for some $C$ independent of $r$ and $k$ if $r$ and $k^{-1}$ are sufficiently
small. Therefore
  \begin{align} \label{E702}
|B_{r^{-1}d_i}(y_i,1)|_{\bar g_i}\le N_k |B_{r^{-1}d_i}(x_{k,i},3k^{-1})|_{\bar
g_i} \le CN_kk^{-m}\le Ck^{l-m}.
  \end{align}

However, $|B_{r^{-1}d_i}(y_i,1)|_{\bar g_i}=r^{-m}|B_{d_i}(y_i,1)|_{ g_i}$,
which by Theorem \ref{T101} again, is greater than a constant
$C=C(m,K,F(2D),D,V_0)$. If we let $\ep \to 0$, then we get a contradiction from
\eqref{E702}.

Conversely, for any point $y \in B(p_{\infty},D)$ such that there exists a
sequence $r_k \to 0$ satisfying
  \begin{align*}
d_{GH}\left\{ \left( B_{{r_k}^{-1}d_{\infty}}(y, 2), r_k^{-1} d_{\infty}
\right), \left( B_{d_{Euc}}(0, 2),  d_{Euc} \right) \right\} <\ep
  \end{align*}
for any $\ep>0$ if $k$ is sufficiently large. As before, with $r_k$ fixed, we
have
  \begin{align} \label{E703}
    d_{GH}\left\{ \left( B_{r_k^{-1}d_i}(y_{i}, 2), r_k^{-1}d_i \right), \left(
    B_{d_{Euc}}(0, 2),  d_{Euc} \right) \right\} <\ep
 \end{align}
if $i$ is sufficiently large. Now we can apply \cite[Lemma $4.11$]{WZ17} to
conclude that
  \begin{align*}
|B_{r_k^{-1}d_i}(y_i,1)|_{\bar g_i} \ge (1-\Psi(\ep))\omega_m
  \end{align*}
for $\Psi(\ep) \to 0$ if $\ep \to 0$.
By Theorem \ref{T101}, it is clear that for any $s \le 1$,
  \begin{align*}
|B_{r_k^{-1}d_i}(y_i,s)|_{\bar g_i}  \ge (1-\Psi(\ep))\omega_ms^m.
  \end{align*}
In other words, for any $r \le r_k$,
  \begin{align*}
|B_{r^{-1}d_i}(y_i,1)|_{r^{-2} g_i}  \ge (1-\Psi(\ep))\omega_m.
  \end{align*}
From \cite[Corollary $4.8$]{WZ17} which we apply to the metric $r^{-2}g_i$ and
$\bar f_i$, it implies that
  \begin{align*}
d_{GH}\left\{ \left( B_{r^{-1}d_i}(y_i,1), r^{-1}d_i \right), \left(
B_{d_{Euc}}(0, 1),  d_{Euc} \right) \right\} <\Psi(\ep).
  \end{align*}
Note that the above inequality holds uniformly for any $r \le r_k$. By taking
$i \to \infty$, we have
  \begin{align*}
d_{GH}\left\{ \left( B_{r^{-1}d_{\infty}}(y_{\infty},1), r^{-1}d_{\infty}
\right), \left( B_{d_{Euc}}(0, 1),  d_{Euc} \right) \right\} <\Psi(\ep).
  \end{align*}
We can conclude immediately that all tangent cones at $y_{\infty}$ is
$(\mathbb{R}^m,g_{Euc})$.
\end{proof}

In general, we notice that for any fixed $D>0$, the concepts of (weakly)
$k$-Euclidean points in $B(p_{\infty},D)$ are defined indifferently from the
case with a uniform Ricci curvature lower bound (see Definition 0.3, Definition
2.2 and Definitions of $\underline{\mathcal{WE}}_k$ and
$(\mathcal{WE}_k)_{\varepsilon}$ in~\cite{ChCo97}). Therefore, the concepts
involved in proving~\cite[Theorem 2.1]{ChCo97} are parallel to the case of
$\mathcal{N}_m(F,K)$, and the very same proof leads to the following
\begin{prop}[Neligibility of the singular set]\label{prop: zero_singular}
Suppose a sequence $\{(M_i,p_i,g_i,f_i)\}\subset \mathcal{N}_m(F,K)$ converges
to a limit metric space $(X,p_{\infty},d_{\infty})$ in the
pointed-Gromov-Hausdorff topology, together with a limit function
$f_{\infty}$ and a limit measure $\mu_{\infty}$, then
\begin{align}
\mu_{\infty}(\mathcal{S})\ =\ 0.
\end{align}
\end{prop}

In the case of the measured, pointed Gromov-Hausdorff limits of a sequence of
Ricci shrinkers, we of course have the limit probability measure satisfying
$\rho_{\infty}(\mathcal{S})=0$.

\section{Parabolic smoothing of the distance function}

This section contains the analytic core of the paper: the $f$-heat kernel
bounds on manifolds in $\mathcal{N}_m(F,K)$, Theorem \ref{T201}, and their
applications in the parabolic smoothing of the distance functions (Lemma
\ref{L303} and Lemma \ref{L304}). 
\subsection{Heat kernel on manifolds in $\mathcal N_m(F,K)$}

Given a metric measure space $(M,p,g,\mu_f)$ in $\mathcal N_m(F,K)$, note
that the weighted Laplacian operator $\Delta_f$ is self-adjoint with respect to
the measure $\mu_f$. Moreover, we have the following Bochner formula for any
smooth function $u$ on $M$,
\begin{align}\label{E200} 
\frac{1}{2} \df|\na u|^2\ =\ |Hess_u|^2+Rc_f(\na u,\na u)+\la \na \df u,\na u
\ra.
\end{align}
If $u$ is defined on the spacetime $M \times [0,T)$ and satisfies the weighted
heat equation
\begin{align*} 
\square_f u\ \coloneqq\ (\pt-\Delta_f)u\ =\ 0,
\end{align*}
then a parabolic version of \eqref{E200} is
\begin{align}\label{E200a} 
\frac{1}{2}\square_f |\na u|^2\ =\ -|Hess_u|^2-Rc_f(\na u,\na u).
\end{align}

Now we denote the heat kernel by $H(x,y,t)$ or $H_f(x,y,t)$ if we want to
emphasize the role of $f$. The existence and uniqueness of $H$ can be found,
for example in \cite[Theorem $7.7$, Corollary $9.6$]{Gr09}. To apply
\cite[Corollary $9.6$]{Gr09}, we must check the stochastic completeness of
$(M,g,\mu_f)$. By our definition, $(M,g,\mu_f)$ is a smooth $CD(K,\infty)$
space, then there exists a constant $C$ such that
\begin{align*} 
\mu_f(B(p,r))\ \le\ Ce^{Cr^2}.
\end{align*}
The proof of the above inequality can be found in \cite[Theorem
$18.12$]{Vil08}. Then the stochastic completeness follows immediately from
\cite[Theorem $11.8$]{Gr09}.

We have the following upper and lower bound of $H$, see also \cite[Theorem
$1.1$]{WW16}:
\begin{thm}
\label{T201}
Let $(M,p,g,\mu_f)$ be a space in $\mathcal N_m(F,K)$.
For any $D>0$, there exists a constant $C=C(m,F,K,D)>1$ such that
\begin{align}\label{E201} 
\frac{C^{-1}}{\mu_f(B(x,\sqrt t))} \exp{\lc-\frac{d^2(x,y)}{C^{-1}t}\rc}\
\le\ H(x,y,t)\ \le\ \frac{C}{\mu_f(B(x,\sqrt t)}
\exp{\lc-\frac{d^2(x,y)}{Ct}\rc}.
\end{align}
for any $x,\,y \in B(q,D/3)$ with $d(p,q) \le D$ and $0<t<D^2/4$.
\end{thm}
\begin{proof}
The upper and lower bound of the weighted heat kernel $H$ follow essentially
from \cite[Theorem $4.1$, Theorem $4.8$]{KTS96}. In our setting, it is clear
that the Dirichlet form is defined as
\begin{align*} 
\mathcal E(u,v)=\int \la \na u,\na v \ra\ \dmu_f.
\end{align*}
Then the Markov semigroup $(P_t)_{t \ge 0}$ satisfies for any $t>0$ and $u \in
L^2(M,\mu_f)$,
\begin{align*} 
\frac{\text{d}P_t u}{\text{d}t}\ =\ \Delta_f P_tu.
\end{align*}
Since we have the local volume doubling property (\ref{eqn: volume_doubling})
and $L^2$-Poincar\'e inequality (Proposition \ref{T:poincare}), then conclusion
follows immediately.
\end{proof}

\begin{cor} \label{C201a}
With the same conditions as those in Theorem \ref{T201}, there exists a uniform 
constant $C=C(m,F,K,D)>1$ such that
\begin{align*} 
\int_{B(q,D/3) \backslash B(q,r)}H(q,y,t)\ \dmu_f(y)\ \le\ Cr^{-2}t
\end{align*}
for any $r \le D/10$ and $t \le D^2/4$.
\end{cor}
\begin{proof}
By computation
\begin{align*} 
&\int_{B(q,D/3) \backslash B(q,r)}H(q,y,t)\ \dmu_f(y)\\
 \le\ &\sum_{k=N_0}^{N_1} \int_{B(q,2^k\sqrt t)\ \backslash B(q,2^{k-1}\sqrt t)}
 H(q,y,t)\ \dmu_f(y) \\
\le\ &\sum_{k=N_0}^{N_1} Ce^{-C^{-1}4^k}\frac{\mu_f(B(x,2^k\sqrt
t)}{\mu_f(B(x,2^{k-1}\sqrt t)} \\
\le\ &C\sum_{k=N_0}^{\infty} e^{-C^{-1}4^k} \le Cr^{-2}t
\end{align*}
where $N_0=\left \lceil\log_2 \frac{r}{\sqrt t}\right\rceil$ and
$N_1=\left\lceil\log_2{\frac{D}{3\sqrt t}}\right\rceil$. Here we have used the
elementary inequality
\begin{align*} 
\sum_{k=0}^{\infty} e^{-4^kl}\ \le\ Cl^{-1}
\end{align*}
for any $l>0$.
\end{proof}

We need the following Li-Yau gradient estimate from \cite[Theorem
$3.1,\,(a)$]{K16}:
 \begin{thm}
\label{T201a}
With the same conditions, there exists a constant $C=C(m,F,K,D)>1$ such that 
\begin{align*} 
C^{-1}\frac{|\na u|^2}{u^2}-\frac{u_t}{u}\ \le\ \frac{C}{t}
\end{align*}
on $B(q,D) \times [0,D^2]$.
\end{thm}

Next we need the following Harnack inequality which is a special case of
\cite[Theorem $3.1$]{WW16}:
 \begin{thm}
\label{T201a}
With the same conditions,  there exists a constant
$C=C(m,K,F,D)>1$ such that
\begin{align*} 
\sup_{B(q,r/2)} u(\cdot,r^2/2)\ \le\ Cu(q,r^2)
\end{align*}
where $u$ is a positive solution of the weighted heat equation on $B(q,r)
\times [0,r^2]$.
\end{thm}

For later applications, we now show the following gradient estimate:
\begin{lem}[Cheng-Yau estimate]
\label{L200}
Let $(M,p,g,\mu_f)$ be a space in $\mathcal N_m(F,K)$. Consider a smooth
function $u$ on $B(q,r)\times [s,s-r^2]$ with $r\le D$ which satisfies the
weight heat equation and is bounded. Then there exists a constant
$C=C(m,K,F,D)>0$ such that for any $r \le D$,
\begin{align*} 
\underset{B(q,r/2) \times [s-r^2/2,s]}{|\nabla u|}\ \le\ Cr^{-1}\underset{B(q,r)
\times [s-r^2,s]}{\text{osc}}\,u.
\end{align*}
\end{lem}

\begin{proof}
Without loss of generality, we assume that $u \ge 0$. We choose a cutoff
function $\psi$ on $\mathbb{R}$ such that $\psi =1$ on $(-\infty,r]$ and
$\psi=0$ on $[2r,\infty]$. Moreover, we assume that $|\psi'| \le C(n)r^{-1}$.
We set $\eta(x,t)=\psi(d(q,x))\psi(s-t)$.
Multiplying both sides of $\square_f u=0$ by $\eta^2 u$ and integrating by
parts, we obtain
\begin{align*}
\iint |\nabla (\eta u)|^2\ \dmu_f\text{d}t\ \le\ &\iint (|\nabla
\eta|^2+\eta_t^2/2)u^2\ \dmu_f\text{d}t-\left. \int u^2\eta^2/2\ \dmu_f
\right|_{t=s}
\\
\le\ &Cr^{-2} \iint_{B(q,r) \times [s-r^2,s]} u^2\ \dmu_f\text{d}t
\end{align*}
That is,
\begin{align*}
\iint_{B(q,r/2)\times [s-r^2/2,s]} |\nabla u|^2\ \dmu_f\text{d}t\ \le\
C\mu_f(B(q,r)) \underset{B(q,r) \times [s-r^2,s]}{\text{osc}}\,u^2.
\end{align*}

On the other hand, by computation,
\begin{align*}
\square_f |\na u|^2\ =\ -2|Hess_u|^2-2Rc_f(\na u,\na u)\ \le\ 2K|\na
u|^2
\end{align*}
or
\begin{align*}
\square_f\lc e^{-2Kt}|\na u|^2\rc\ \le\ 0.
\end{align*}

Now we apply the Moser iteration on $e^{-2Kt}|\na u|^2$, see \cite[Proposition
$2.7$]{WW16}, that
\begin{align*}
\underset{B(q,r/4) \times [s-r^2/4,s]}{|\nabla u|^2}\ \le\
\frac{C}{r^2\mu_f(B(q,r/2))}\iint_{B(q,r/2) \times [s-r^2/4,s]} |\nabla
u|^2\ \dmu_f\text{d}t.
\end{align*}

Therefore,
\begin{align*}
\underset{B(q,r/4) \times [s-r^2/4,s]}{|\nabla u|^2}\ \le\
\frac{C\mu_f(B(q,r))}{r^2\mu_f(B(q,r/2))} \underset{B(q,r) \times
[s-r^2,s]}{\text{osc}}\,u^2\ \le\ Cr^{-2} \underset{B(q,r) \times
[s-r^2,s]}{\text{osc}}\,u^2
\end{align*}
by the volume doubling.
\end{proof}

Now we prove the existence of a local cut-off function, see also \cite[Theorem
6.33]{ChCo96}.
\begin{thm}
\label{T201b}
Let $(M,p,g,\mu_f)$ be a space in $\mathcal N_m(F,K)$. For
any $D>0$ and $q$ with $d(p,q) \le D$, there exists a smooth cutoff function
$\phi$ which is supported in $B(q,r)$ and $\phi=1$ on $B(q,r/2)$ such that
\begin{align*} 
r|\nabla \phi|+r^2|\Delta_f \phi|\ \le\ C
\end{align*}
for some $C=C(m,K,F,D)>0$ and $r \le D/3$.
\end{thm}

\begin{proof}
Consider the function $u(x,t)=H(x,q^*,t)$, then it follows from \eqref{E201}
that there exist $C_0=C_0(m,K,F(2D),D)>1$ and $C_1=C_1(m,K,F(2D),D)>1$ such
that
\begin{align*}
\begin{cases}
u \ge C_0^{-1}a &\quad \text{on}\, B(q^*,\rho/2) \times [\rho^2/3, \rho^2],\\
u \le C_0^{-1}a/2 &\quad \text{on}\,  B(q^*,2C_1\rho) \backslash B(q^*,C_1\rho)
\times [\rho^2/3, \rho^2],\\
u \le Ca &\quad \text{on}\,  B(q^*,2C_1\rho)  \times [\rho^2/4, \rho^2],\\
\end{cases}
\end{align*}
 where $a=v^{-1}_f(B(q^*,\rho))$l. Here we require that all sets considered are
 contained in $B(q,D/6)$.

 Now it follows from Lemma \ref{L200} that 
\begin{align*} 
|\na u|\ \le\ C{\rho}^{-1}a
\end{align*}
on $B(q^*,C_1r) \times [\rho^2/3,\rho^2]$.

Now we take a nonnegative cut-off function $\eta$ supported in
$[\rho^2/2,2{\rho}^2/3]$ such that $|\eta'| \le C{\rho}^{-2}$ and
\begin{align*} 
\int_{\rho^2/2}^{2{\rho}^2/3} \eta\ \ge\ C^{-1}\rho^2.
\end{align*}

Then we define a function
\begin{align*} 
\psi(x)\ =\ \frac{C_0}{a{\rho}^2} \int_{\rho^2/2}^{2r^2/3} u(x,t) \eta(t)\
\text{d}t.
\end{align*}

Then we have 
\begin{align*}
\begin{cases}
\psi&\ge\ 1 \quad \text{on}\ B(q^*,\rho/2),\\
\psi&\le\ 1/2 \quad \text{on}\  B(q^*,2C_1\rho) \backslash B(q^*,C_1\rho),\\
|\na \psi|&\le\ C\rho^{-1} \quad \text{on}\  B(q^*,C_1\rho).\\
\end{cases}
\end{align*}
Moreover,
\begin{align*} 
\Delta_f\psi(x)\ =\ &\frac{C_0}{a\rho^2} \int_{\rho^2/2}^{2{\rho}^2/3} \Delta_f
u(x,t) \eta(t)\text{d}t \\
=\ &\frac{C_0}{a\rho^2} \int_{\rho^2/2}^{2\rho^2/3} \pt u(x,t) \eta(t)\text{d}t
\\
=\ &-\frac{C_0}{a\rho^2} \int_{\rho^2/2}^{2\rho^2/3}  u(x,t)
\pt\eta(t)\text{d}t,
\end{align*}
therefore on $B(q^*,C_1r)$,
\begin{align*} 
\rho^2|\Delta_f\psi|\ \le\ C.
\end{align*}
Now we construct a smooth nondecreasing function $F(t)=0$ if $t\le 1/2$ and
$F(t)=1$ if $t \ge 1$. Then by considering the composite function $F(\psi(x))$
we have proved that for any $B(q^*,r) \subset B(q,D)$, there exists a cutoff
function $\phi^*$ supported in $B(q^*,r)$ such that $\phi_*=1$ on
$B(q^*,C_1^{-1}r)$ and
 \begin{align*} 
r|\na \phi_*|+r^2|\Delta_f\phi_*|\ \le\ C.
\end{align*}

The rest proof is a standard covering argument. By the local volume doubling,
there exists an integer $N=N(m,K,F(2D))>1$ such that we can find
$q_1,q_2,\cdots,q_N \in M$ such that
 \begin{align*} 
B(q,1/2) \subset \bigcup_{i=1}^N B(q_i,C_1^{-1}r).
\end{align*}
Then the function $\phi \coloneqq F(\sum_{i=1}^N \phi_i)$ will satisfy all
conditions.

Now the theorem follows from a standard argument by \eqref{eqn:
volume_doubling}.
\end{proof}

\begin{remark}
In \cite[Lemma 1.5]{WZ17}, the same conclusion is proven using Green's function.
\end{remark}

Now we need the following space-time control of the Hessian term of a heat
equation solution:
\begin{lem}
\label{L200a}
With the same assumptions as in Lemma \ref{L200}, there exists
$C=C(m,K,F,D)>0$ such that
\begin{align*} 
\int_{s-r^2/4}^s \aint_{B(q,r/4)} |Hess_u|^2\ \dmu_f\text{d}t\ \le\
Cr^{-2}\lc\underset{B(q,r) \times [s-r^2,s]}{\text{osc}}\,u \rc^2.
\end{align*}
\end{lem}

\begin{proof}
We have as before
\begin{align}\label{E202}
\square_f |\na u|^2\ =\ -2|Hess_u|^2-2Rc_f(\na u,\na u)\ \le\ 
-2|Hess_u|^2+2K|\na u|^2.
\end{align}

We choose a nonincreasing cutoff function $\eta$ on $\mathbb{R}$ such that
$\eta(x)=1$ if $x \le 1/2$ and $\eta(x)=0$ if $x \ge 1$. Let $\phi$ be a cutoff
function constructed in the last theorem such that $\phi=1$ on $B(q,r/2)$ and
is supported in $B(q,r)$. We also set $\psi(x,t) \coloneqq \phi(x)\eta(-tr^2)$.
By multiplying both sides of \eqref{E202} by $\psi$ and integrating, we have
\begin{align*}
\iint 2|Hess_u|^2 \psi\ \dmu_f\text{d}t\ \le\ &\iint -\psi \square_f
|\na u|^2+2K|\na u|^2\psi\ \dmu_f\text{d}t \\
=\ &\iint (\pt+\df) \psi|\na u|^2+2K|\na u|^2\psi\ \dmu_f\text{d}t  -\left.
\int |\na u|^2 \psi\ \dmu_f \right|_{t=s} \\
\le\ &Cr^{-2} \iint_{B(q,r/2)\times [s-r^2/2,s]} |\na u|^2\ \dmu_f\text{d}t.  
\end{align*}

Then the conclusion follows from Lemma \ref{L200} and the local volume
doubling.
\end{proof}

For a closed set $X \subset M$ and $0<r_0<r_1$, the annulus $A_{r_0,r_1}(X)$ is
defined as $T_{r_1} \backslash T_{r_0}$, where $T_r(X)$ is the $r-$tubular
neighborhood of $X$. Then by using Theorem \ref{T201a} and a similar argument
in \cite[Lemma $2.6$]{CN12} we have

\begin{cor}
\label{C201}
Let $(M,p,g,\mu_f)$ be a space in $\mathcal N_m(F,K)$. For
any $D>0$, $0<10r_0<r_1<D/10$ and a compact set $X \subset B(q,D/10)$ with
$d(p,q) \le D$, there exists a smooth nonnegative cutoff function $\phi$ such
that for some constant $C=C(m,K,F,D)>0$,
\begin{enumerate}
\item $\phi=1$ on $A_{3r_0,r_1/3}(X)$ and $\phi=0$ on $M \backslash
A_{2r_0,r_1/2}(X)$.
\item $r_0|\na \phi|+r_0^2|\df \phi|\ \le\ C$ on $A_{2r_0,3r_0}(X)$.
\item $r_1|\na \phi|+r_1^2|\df \phi|\ \le\ C$ on $A_{r_1/3,r_1/2}(X)$.
\end{enumerate}
\end{cor}

Next we prove the following mean value inequality which is similar to the Lemma
$2.1$ of \cite{CN12}.
\begin{lem}
\label{L202}
Let $(M,p,g,\mu_f)$ be a space in $\mathcal N_m(F,K)$. For
any $D>0$ and $q$ with $d(p,q) \le D$, there
exists a constant $C=C(m,K,F,D)>1$ such that the following holds. If
$u_t=u(x,t)$ is nonnegative continuous function on $M \times [0,r^2]$ with
compact support on each time slice in $B(q,D/5)$, $r \le D/10$ and $\square_f u
\ge -c_0$ in the distribution sense, then
 \begin{align*} 
\aint_{B(x,r)} u_0\ \dmu_f\ \le\ C (u_{r^2}(x)+c_0r^2).
\end{align*}
\end{lem}

\begin{proof}
We fix $x$ and $r$, then for any $t \in [0,r^2)$, we have
 \begin{align}\label{E205} 
&\partial_t\int  H(x,y,r^2-t)u(y,t)\ \dmu_f(y)  \notag\\
=\ &\int \partial_tHu+H\partial_t u\ \dmu_f \notag\\
=\ &\int -\df Hu+H\df u+H\square_f u\ \dmu_f\notag\\
=\ &\int H\square_f u\ \dmu_f  \notag \\
\ge\ &-c_0\int H\ \dmu_f\ =\ -c_0.
\end{align}

As $t \to r^2$, it follows from the definition of $H$ that
 \begin{align*} 
\lim_{t \nearrow r^2}\int  H(x,y,r^2-t)u(y,t)\ \dmu_f(y)\ =\ u(x,r^2).
\end{align*}

From the heat kernel lower bound \eqref{E201}, we have
 \begin{align*} 
\int  H(x,y,r^2)u_0(y)\ \dmu_f(y)\ \ge\ C \aint_{B(x,r)}u_0(y)\ \text{d}
\mu_f(y).
\end{align*}

The proof is complete if we integrate \eqref{E205} from $0$ to $r^2$.
\end{proof}

In particular, if $u$ is independent of $t$, we have

\begin{cor}
\label{C202}
Let $(M,p,g,\mu_f)$ be a space in $\mathcal N_m(F,K)$. For any $D>0$ and $q\in
B(p,D)$, there exists a constant $C=C(m,K,F,D)>1$ such that the following
holds. If $u(x)$ is nonnegative continuous function on $M$ with compact support
in $B(q,D/5)$, $r \le D$ and $\df u \le c_0$, then
 \begin{align*} 
\aint_{B(x,r)} u\ \dmu_f\ \le\ C (u(x)+c_0r^2).
\end{align*}
\end{cor}
\noindent 

\subsection{Smoothing the distance function}
In this subsection, we fix a manifold $(M^m,p,g,f)\in \mathcal
N_m(F,K)$ and two points $q',q$ in $B(p,D/2)$ with $d(q',q)=d$. Recall that the
excess function of $q'$ and $q$ is defined as
 \begin{align*} 
e(x)\ :=\ d(x,q')+d(x,q)-d(q',q).
\end{align*}
We also set $d^-(x)=d(q',x)$ and $d^+(x)=d-d(q,x)$.

Now we have
\begin{thm}
\label{T301}
Assume that $d \le 1$ and a constant $0<\epsilon<1$, if $x \in A_{\ep
d,2d}(\{q',q\})$ satisfies $e(x) \le r^2d \le \bar r^2(m,\ep)d$, then
\begin{align*} 
\aint_{B(x,rd)} e\ \dmu_f\ \le\ C_{Ex}r^2d
\end{align*}
for some $C_{Ex}=C_{Ex}(m,K,F(2D))>0$.
\end{thm}

\begin{proof}
By Corollary\ref{C201} with $X=\{q',q\}$, there exists a cutoff function $\phi$.
If $u=\phi e$, then
\begin{align*} 
\df u\ =\ \df \phi  e+\phi \df e+2\la \na \phi,\na e \ra\ \le\ |\df \phi |e+\phi
\df e+4|\na \phi |\ \le\ \frac{C}{d}
\end{align*}
where we have used Theorem \ref{T102} and $d\le 1$ for the last inequality.
Then the theorem follows from Corollary \ref{C202}.
\end{proof}

Similar to \cite[Section $2$]{CN12}, we evolve the distance functions to $q'$
and $q$ by the weighted heat equation. For a fixed $\delta>0$, by using
Corollary \ref{C201}, there exists a cutoff function $\phi$ such that $\phi=1$
on $M_{\delta d/4,8d}$ and $\phi=0$ on $M\backslash M_{\delta d/16,16d}$ where
\begin{align*} 
M_{r,s} \coloneqq A_{rd,sd}(q') \cap A_{rd,sd}(q).
\end{align*}
Then we define $h_t^{\pm}$ and $e_t$ to be solutions to the equation
$\square_fh_t^{\pm}=0$ and $\square_fe_t=0$ with initial values $h_0^{\pm}=\phi
d^{\pm}$ and $e_0=\phi e$, respectively.

\begin{lem}
\label{L301}
There exists a constant $C=C(m,K,F(2D),\delta)>0$ such that
\begin{align*} 
\df e_t\,,\df h_t^-\,,-\df h_t^+\ \le\ \frac{C}{d}.
\end{align*}
\end{lem}

\begin{proof}
We only prove the conclusion for $e_t$, others are similar.
As before, we have
\begin{align*} 
\df e_0\ =\ \df (\phi e)\ \le\ \frac{C}{d}. 
\end{align*}

Moreover, for any $t>0$,
\begin{align*} 
\df e_t\ =\ \int \Delta_{f,x} H(x,y,t)  \phi(y)e(y)\ \dmu_f(y)\ =\ \int
\Delta_{f,y} H(x,y,t)  \phi(y)e(y)\ \dmu_f(y)\ \le\ \frac{C}{d}.
\end{align*}
\end{proof}

\begin{lem}
\label{L302}
There exists a constant $C=C(m,K,F(2D),\delta)>0$ such that for any $\ep \le
\bar{\ep}(m,\delta)$ and $x \in M_{\delta/2,4}$, the following holds for each
$y \in B(x,10\ep d)$, \begin{enumerate}[label=(\roman*)]
\item $|e_{\ep^2d^2}(y)|\ \le\ C(\ep^2d+e(x)).$
\item$ |\pt e_{\ep^2d^2}|(y)+|\df e_{\ep^2d^2}|(y)\ \le\
C(\frac{1}{d}+\frac{e(x)}{\ep^2d^2}).$
\item $\aint_{B(y,\ep d)}|Hess_{e_{\ep^2d^2}}|\ \dmu_f\ \le\
C(\frac{1}{d}+\frac{e(x)}{\ep^2d^2})^2$. 
\item $|h^{\pm}_{\ep^2 d^2}-d^{\pm}|(x)\ \le\ C(\ep^2d+e(x)).$
\item $|\na h^{\pm}_{\ep^2d^2}|(x)\ \le\ 1+C\ep^2 d^2.$
\end{enumerate}
\end{lem}

\begin{proof}
From Lemma \ref{L301}, we have
\begin{align*} 
e_t(x)\ =\ e_0(x)+\int_0^t \df e_s\ \text{d}s\ \le\ e(x)+\frac{Ct}{d}
\end{align*}
and hence
\begin{align*} 
e_t(x)\ \le\ e_0(x)+C\ep^2d
\end{align*}
for any $t \in [\ep^2d^2,100\ep^2d^2]$. Then it follows from Theorem
\ref{T201a} that if $y \in B(x,10d\ep)$,
\begin{align*} 
e_{\ep^2d^2}(y)\ \le\ C(\ep^2d+e(x))
\end{align*}
and the (i) follows. (ii) follows from Lemma \ref{L200} and (iii) follows from
Lemma \ref{L301} and the gradient estimate Theorem \ref{T201a}. In addition,
(iv) follows from Lemma \ref{L200a}. Indeed, since we have
\begin{align*} 
\int_{\ep^2d^2/2}^{\ep^2d^2} \aint_{B(y,\ep d)} |Hess_{e_t}|^2\
\dmu_f \text{d}t\ \le\ C\lc\frac{1}{d}+\frac{e(x)}{\ep d}\rc^2.
\end{align*}
Therefore such $r$ must exists. Next (v) follows from an identical proof in
\cite[Lemma $2.3$]{CN12}. Finally, (vi) follows the same as \cite[Lemma
$2.17$]{CN12} by using Corollary \ref{C201a}.
\end{proof}

Recall that an $\ep$-geodesic between $q'$ and $q$ is a unit speed curve
$\sigma$ such that $||\sigma|-d(q',q) |\le \ep $. In particular, it implies that
for any $x \in \sigma$, $e(x) \le \ep^2d^2$.

\begin{lem}
\label{L303}
There exists a constant $C_3=C_3(m,F,K,D,\delta)>0$ such that for any $\ep
\le \bar{\ep}(m,\delta)$, $x \in M_{\delta/2,4}$ with $e(x) \le \ep^2d$, and
any $\ep$-geodesic $\sigma$ connecting $q'$ and $q$, there exists $c \in
[1/2,2]$ with \begin{enumerate}[label=(\roman*)]
\item  $|h^{\pm}_{\ep^2 d^2}-d^{\pm}|(x)\ \le\ C_3\ep^2d.$
\item $ \aint_{B(x,\ep d)}||\na h^{\pm}_{c\ep^2d^2}|^2-1|\ \dmu_f\ \le\
C_3\ep.$
\item $\int_{\delta d}^{(1-\delta)d} \left(\aint_{B(x,\ep d)}|\na
h^{\pm}_{c\ep^2d^2}|^2-1|\ \dmu_f\right)\ \le\ C_3\ep^2d.$
\item $\int_{\delta d}^{(1-\delta)d}
\left(\aint_{B(\sigma(s),\ep d)}|Hess_{h^{\pm}_{c\ep^2d^2}}|\
\dmu_f\right)\ \le\ C_3d^{-2}.$
\end{enumerate}
\end{lem}
\begin{proof}
The proof is the same as \cite[Theorem $2.19$]{CN12}.
\end{proof}

Similar to  \cite[Theorem $2.20$]{CN12}, we also have

\begin{lem}
\label{L304}
There exists a constant $C_4=C_4(m,F,K,D,\delta)>0$ such that for any $x \in
M_{\delta,2}$ and $\delta\le s <t \le d_x=d(q',x)$, the following estimates
hold, \begin{enumerate}[label=(\roman*)]
\item $\int_{\delta}^{d_x} ||\na h^-_{r^2}|^2-1|\ \le\ \frac{C_4}{d}(e(x)+r^2)$.
\item $\int_{\delta}^{d_x} |\la \na h^-_{r^2},\na d^- \ra-1|\ \le\
\frac{C_4}{d}(e(x)+r^2)$.
\item $\int_s^t | \na h^-_{r^2}-\na d^-|\ \le\ \frac{C_4\sqrt{t-s}}{\sqrt
d}(\sqrt{e(x)}+r)$.
\end{enumerate}
\end{lem}

\section{Convexity of the regular part in Gromov-Hausdorff limits}

In this section we prove our main improvement of the previous structural
results in~\cite{LLW17} about the pointed-Gromov-Hausdorff limits of manifolds
in $\mathcal{N}_m(F,K)$: we will show that the regular part, as defined in
Section 2.3, is both weakly convex and almost everywhere convex with respect to
the limit measure. In conjunction with the regularity improvements obtained in
\cite{LLW17}, we also show that the regular part on a pointed-Gromov-Hausdorff
limit of manifolds in $\mathcal{N}_m(F,K)$ is actually a strongly convex open
subset. Results about Ricci shrinkers are summarized in Theorem~\ref{thm: RS}.

\subsection{Gromov-Hausdorff distance between nearby metric balls} 
In this subsection, we prove a Gromov-Hausdorff distance control of nearby
geodesic balls of the same size. The proof follows from the original idea
in~\cite[Section 3]{CN12}, but since the Ricci lower bound, the basic
assumption underlying essentially everywhere of their arguments, is unavailable
for manifolds in $\mathcal{N}_m(F,K)$, we have to rework most of the details
there and fit them into our setting.

Fix $\gamma:[0,l]\to M$, a minimal geodesic of length $l$ and unit speed. Let
$\gamma(0)=q$ and $\gamma(l)=q'$. For $d_q:=d(q,-)$, let $\psi_s$ be the
gradient flow generated by the almost everywhere defined vector field $-\nabla
d_q$. Notice that $\psi_s$ is only defined almost everywhere on $M$, and 
that $\psi_s$ is smooth away from $q$ and the cut locus $\mathcal{C}_q$ of $q$.
Moreover, if $x\not\in \mathcal{C}_q\cup\{q\}$, then the integral curve
$s\mapsto \psi_s(x)$ defines a minimal geodesic of unit speed for $s\in
[0,d(q,x))]$. In this section we fix some $\delta \in (0,1\slash 10)$, and will
control the Gromov-Hausdorff distance between balls centered at nearby points on
$\gamma([\delta l,l-\delta l])$.

Fixing $t\in [\delta l,l-\delta l]$, for each $r\in [0, \delta \slash 10]$, we
consider the following core neighborhood of $\gamma(t)$:
\begin{align*}
H_r^t\ :=\ \left\{y\in B(\gamma(t),r):\ \forall s\in
[0,t-\delta l],\ d(\psi_s(y),\gamma(t-s))\le
\exp\left(C_5\sqrt{s}\right)d(y,\gamma(t))\right\},
\end{align*}
where 
\begin{align*}
C_5=C_5(m,F,K,D,\delta,l):=\left(4(m-1)(\delta
l)^{-1}+2(m-1)\sqrt{K}+2F(2D)+(1-\delta)Kl\right)^{\frac{1}{2}}.
\end{align*}
Intuitively speaking, such neighborhood of $\gamma(t)$ consists of
points in $B(\gamma(t),r)$ that are carried by $\psi_s$ up to a controllable
distance for all $s\le t-\delta l$. On a manifold in $\mathcal{N}_m(F,K)$, we
could in fact conclude that when $r$ is sufficiently small, \emph{almost every}
point of $B(\gamma(t),r)$ are in $H^t_r$:

\begin{lem}\label{lem: core_nbhd}
Fix $\delta\in (0,1\slash 10)$ and $t\in [\delta l,l-\delta l]$. Let
$\gamma:[0,l]\to M$ be a minimal geodesic of unit speed, with $\gamma(0)=q$ and
$\gamma(l)=q'$, then for $r>0$ sufficiently small,
\begin{align}
\mu_f(H_r^t)\ =\ \mu_f(B(\gamma(t),r)).
\end{align}
\end{lem}

\begin{proof}
Let $\iota(x)$ denote the injectivity radius of $x\in M$, and set 
\begin{align*}
r\ \le\ \frac{1}{10}\min\left\{\delta l, \iota(\gamma(t))\right\},
\end{align*}
such that $B(\gamma(t),r)\cap \mathcal{C}_q=\emptyset$, and that for any $x\in
B(\gamma(t),r)$, the minimal geodesic starting from $q$ and passing $x$ can be
extended at least up to $(1-\frac{1}{4}\delta)l$. Since $\gamma$ is a minimal
geodesic, $\gamma([0,l])$ is compact and $\iota$ is positive and continuous on
$M$, we see $r>0$.

For any $x\in B(\gamma(t),r)$, there exists a unique tangent vector $v\in
T_{\gamma(t)}M$ such that $\exp_{\gamma(t)}v=x$. Since $(u,s)\mapsto
\exp_qs(\dot{\gamma}(0)+u\vec{v})$ is a variation by geodesics and thus 
$J(u,s):=\partial_u \exp_qs(\dot{\gamma}(0)+u\vec{v})$ is a Jacobi field along
the geodesic $\gamma_u:s\mapsto \exp_qs(\dot{\gamma}(0)+u\vec{v})$.

On the other hand, since $u\mapsto \exp_q(t-s)(\dot{\gamma}(0)+u\vec{v})$ is a
curve connecting $\gamma(t-s)$ to $\psi_s(x)$, we have 
$d(\psi_s(x),\gamma(t-s))\le \int_0^1|J(u,t-s)|\ \text{d}u$, and it then
suffices to bound $|J|(u,t-s)$ in terms of $|J|(u,t)=|\vec{v}|=d(x,\gamma(t))$.

In order to estimate $|J|$, we notice that
$\frac{\text{d}}{\text{d}t}|J|^2=2Hess_{d_q}(J,J)$, and thus
\begin{align*}
\left|\frac{\text{d}}{\text{d}s}\log |J|^2\right|\ \le\ 2|Hess_{d_q}|.
\end{align*}
Integrating from any $s_1\in [0,t-\delta l]$ to $t$ we see
\begin{align*}
\left|\log\frac{|J|^2(u,t)}{|J|^2(u,t-s_1)} \right|\ \le\
\int_{t-s_1}^t|Hess_{d_q}|(\gamma_u(s))\ \text{d}s\ \le\
\left(\int_{\frac{\delta l}{2}}^{(1-\frac{\delta}{2})l}
|Hess_{d_q}|^2(\gamma_u(s))\ \text{d}s\right)^{\frac{1}{2}}\sqrt{s}.
\end{align*}
Therefore, $\forall s\in [0,t-\delta l]$, integrating $u$ we have
\begin{align}
d(\psi_s(x),\gamma(t-s))\ \le\
\exp \left(s \int_{\frac{\delta l}{2}}^{(1-\frac{\delta}{2})l}|Hess_{d_q}|^2
(\gamma(u))\ \text{d}u\right)^{\frac{1}{2}} d(x,\gamma(t)),
\end{align}
and the following claim guarantees that $B(x,r)\subset H^t_r$, whence the
weighted volume estimate.

\emph{Claim}:
For the given $\gamma$, we have $\int_{\frac{\delta
l}{2}}^{(1-\frac{\delta}{2})l} |Hess_{d_p}|^2(\gamma(s))\ \text{d}s \le
4C_5^2$, with $D:= \sup_{t\in [0,l]}d(p,\gamma(t))$.

\noindent \emph{Proof of claim:}
By Theorem~\ref{T102} we have
\begin{align*}
\Delta_fd_q\ \le\ \frac{m-1}{d_q}+(m-1)\sqrt{K}+F(2D)\quad\text{and}\quad
\Delta_fd_{q'}\ \le\ \frac{m-1}{d_{q'}}+(m-1)\sqrt{K}+F(2D).
\end{align*}

Since the function $(d_q+d_{q'})$ attains a smooth minimum on
$\gamma$, we see $\Delta_f(d_q+d_{q'})(\gamma(t))\ge 0$; we also notice 
that $\nabla d_q(\gamma(t))=-\nabla d_{q'}(\gamma(t))$. Therefore, we have
\begin{align*}
\Delta_fd_q(\gamma(t))\ \ge\ -\Delta_fd_{q'}(\gamma(t))\ \ge\ 
-\frac{m-1}{d_{q'}(\gamma(t))}-(m-1)\sqrt{K}-F(2D),
\end{align*} and thus
\begin{align*}
\max\left\{|\Delta_fd_q|(\gamma(\delta l)),|\Delta_f
d_q|(\gamma(l-\delta l))\right\}\ \le\ \frac{2(m-1)}{\delta
l}+(m-1)\sqrt{K}+F(2D).
\end{align*}

Plugging $u=d_q$ in to the Bochner formula \eqref{E200}, we immediately have:
\begin{align*}
\begin{split}
\forall t\in [\frac{\delta l}{2},l-\frac{\delta}{2} l],\quad
0\ =\ \Delta_f|\nabla d_q|^2(\gamma(t))\ 
\ge\
2|Hess_{d_p}|^2(\gamma(t))+2\partial_t\left(\Delta_fd_q(\gamma(t))\right)-K.
\end{split}
\end{align*}
Therefore, integrating $t\in [\frac{\delta l}{2},l-\frac{\delta}{2} l]$ we
obtain
\begin{align}
\begin{split}
\int_{\frac{\delta l}{2}}^{(1-\frac{\delta}{2}) l}|Hess_{d_q}|^2\ \text{d}t\ 
\le\ &\frac{4(m-1)}{\delta l}+2(m-1)\sqrt{K}+2F(2D)+(1-\delta)Kl,
\end{split}
\label{eqn: Hessian_on_line}
\end{align}
whence the desired $L^2$-estimate.
\end{proof}

From the proof of the proposition, we could clearly see that $H^t_r$ is
determined by the specific $M$ and $\gamma$, rather than a uniform neighborhood
that we would like to get. In fact, it is impossible to get such neighborhood
in a uniform way; however, in the sequel, we will see that there is a
sufficiently large (in volume) set, which is not necessarily a neighborhood of
$\gamma$, but which resembles the key property of $H_r^t$:
the gradient flow lines of $-\nabla d_q$ with initial data in this set does
\emph{not} spread too far away from $\gamma$. Moreover, this set is defined
analytically and its properties depend on the estimates uniformly.

Define
\begin{align*}
\mathcal{B}_s^t(\alpha, r)\ :=\ \{z\in B(\gamma(t),r):\ \forall u\in [0,s l],\
\psi_u(z)\in B(\gamma(t-u),(1+\alpha)r)\}.
\end{align*}

Clearly, $\mathcal{B}_0^t(\alpha, r)=B(\gamma(t),r)$, since $\psi_0$ is the
identity map; by the continuity of $-\nabla d_q$ outside the cut-locus of $q$,
we know that there is a small $\varepsilon>0$ such that, 
\begin{align}
\forall s\in [0,\varepsilon l],\quad
\frac{\mu_f(\mathcal{B}_s^t(\alpha, r))}{\mu_f(B(\gamma(t-s),r))}\ \ge\
\frac{1}{2}.
\label{eqn: volume_estimate}
\end{align}
Clearly, this $\varepsilon$ may vary from one specific manifold to another, it
may also depend on $r$. 

However, let this $\varepsilon$ be chosen as the maximal possible value that
satisfies (\ref{eqn: volume_estimate}), and we will show its irrelevance of
specific manifolds and $r$ provided $r\le r_0$, some fixed constant.

 Now let $c_s^t$ be the characteristic function of
$\mathcal{B}_s^t(\alpha, r)\times \mathcal{B}_s^t(\alpha, r)$ in
$B(\gamma(t),r)\times B(\gamma(t),r)$, then for any $s\in [0,\varepsilon l]$
and $\eta\in (0,1\slash 2)$, we define with the $c\in [\frac{1}{2},2]$ obtained
in Lemma~\ref{L303},
\begin{align}
\mathcal{F}^r_{s}(x,y)\ :=\ &\int_0^{s}c_u^t(x,y)
\left(\int_{\gamma_{\psi_u(x),\psi_u(y)}}|Hess_{h_{cr^2}}|\right)\ \text{d}u,
\label{eqn: Fxy}\\
%\text{and}\quad 
I_{s}^t(r)\ :=\ &\aint_{B(\gamma(t),r)\times
B(\gamma(t),r)}\mathcal{F}_{s}^r(x,y)\ \dmu_f(x)\dmu_f(y), 
\label{eqn: int_I}
\end{align}
\begin{align}\label{eqn: Tr}
T_{\eta}^r := \left\{x\in B(\gamma(t),r):\ e_{q,q'}(x)\le C_{Ex}
r^2(\eta\delta l)^{-1},\ \text{and}\ \aint_{
B(\gamma(t),r)}\mathcal{F}_{\varepsilon l}^r(x,y)\ \dmu_f(y)\ \le\
\eta^{-1}I_{\varepsilon l}^t(r)\right\},
\end{align}
where $C_{Ex}$ is the consant in Theorem~\ref{T301}, and for each $x\in
T_{\eta}^r$ we define
\begin{align}\label{eqn: Trx}
T_{\eta}^r(x) := \left\{y\in B(\gamma(t),r):\
e_{q,q'}(y)\le C_{Ex} r^2(\eta\delta l)^{-1},\ \text{and}\
\mathcal{F}_{\varepsilon l}^r(x,y)\ \le\ \eta^{-2}I_{\varepsilon
l}^t(r)\right\}.
\end{align}
By the excess function estimate in Theorem~\ref{T301} and Chebyshev's
inequality, we have
\begin{align}
\frac{\mu_f(T_{\eta}^r)}{\mu_f(B(\gamma(t),r))}\ \ge\
1-2\eta,\quad\text{and}\quad \forall x\in T_{\eta}^r,\
\frac{\mu_f(T_{\eta}^r(x))}{\mu_f(B(\gamma(t),r))}\ \ge\ 1-2\eta.
\label{eqn: Chebyshev}
\end{align}
Notice that these estimates are \emph{uniform}. Now we come to the following
\begin{lem}
Fix $\alpha,\xi\in (0,1\slash 20)$ and $\eta \in (0,1\slash 100)$. 
There is an $\varepsilon_0 = \varepsilon_0(\eta\ |\ \alpha, m,F,K,D,\delta)$
and a $C_6=C_6(m,F,K,D,\delta)$ such that whenever $\varepsilon <
\varepsilon_0$, then for any fixed $r\in (0,\delta l\slash 10)$, once
(\ref{eqn: volume_estimate}) holds on $[t-\varepsilon l,t]$, then $\forall s\in
[t-\varepsilon l,t]$, $\forall x_1\in T_{\eta}^r$ and $\forall x_2\in
T_{\eta}^r(x_1)\cap \mathcal{B}_s^t(\alpha,\xi r)$,
\begin{align}
|d(\psi_s(x_1),\psi_s(x_2))-d(x_1,x_2)|\ \le\
C_6\eta^{-2}r \sqrt{s\slash l}.
\label{eqn: distance_estimate}
\end{align}
Moreover, $x_1\in \mathcal{B}_s^t(\alpha, r)$ for all $s\in [0,\varepsilon l]$.
\label{lem: distance_estimate}	
\end{lem}
\begin{proof}
Fix any $x_1\in T_{\eta}^r\backslash \mathcal{C}_q$ and denote
\begin{align*}
\varepsilon(x_1):=\ \min\left\{1,\sup\{s\le \varepsilon l:\ \forall u\in
[0,s],\ \psi_{u}(x_1)\in B(\gamma(t-u),2r)\}\right\}.
\end{align*} 
Clearly, when $s\le \varepsilon(x_1)$, $x_1\in \mathcal{B}_s^t(\alpha, r)$;
moreover, $\mathcal{B}_s^t(\alpha, \xi r)\subset \mathcal{B}_s^t(\alpha, r)$.
Therefore $c_s^t(x_1,x_2)=1$. By the continuity of the mapping $u\mapsto
\psi_u(x_1)$, we also see that
\begin{align}
\psi_{\varepsilon(x_1)}(x_1)\not\in B(\gamma(t-\varepsilon(x_1)),\frac{3}{2}r).
\label{eqn: away}
\end{align}
 We will show that $\varepsilon(x_1)= \varepsilon l$ for suitably chosen
$\varepsilon_0$. Now for any $x_2\in (T_{\eta}^r(x_1)\cap
\mathcal{B}_s^t(\alpha, r)) \backslash \mathcal{C}_q$ fixed, we let $\sigma_1$
and $\sigma_2$ denote the integral curves of $-\nabla d_q$ starting from $x_1$
and $x_2$, respectively.
These are minimal geodesics, and integrating (3.6) in~\cite[Lemma 3.4]{CN12}
we get for any $s\le \varepsilon(x_1)$, 
\begin{align}
\begin{split}
|d(\psi_s(x_1),\psi_s(x_2))-d(x_1,x_2)|\ \le\ &\int_0^s|\nabla
h_{cr^2}-\nabla d_q|(\sigma_1(u))\ \text{d}u\\
&+\int_0^s|\nabla h_{cr^2}-\nabla d_q|(\sigma_2(u))\
\text{d}u+\mathcal{F}_{s}^r(x_1,x_2).
\end{split}
\label{eqn: distance_1}
\end{align}
We now estimate each term in the right-hand side of (\ref{eqn: distance_1}). 
By Lemma~\ref{L304} and
the choice of $x_1,x_2$, we see for $i=1,2$,
\begin{align}
\begin{split}
\forall s\in [0,\varepsilon(x_1)],\quad \int_0^s|\nabla h_{cr^2}-\nabla
d_q|(\sigma_i(u))\ \text{d}u\ \le\ 2C_4\eta^{-\frac{1}{2}}
r\sqrt{s\slash l}.
\end{split}
\label{eqn: distance_2}
\end{align}
The last term on the right-hand side of (\ref{eqn: distance_1}) is by
definition bounded by $\eta^{-2}I_s^t(r)$. By the segment inequality of
Theorem \ref{T104} and the definition of $\mathcal{B}_s^t(\alpha, r)$, for any
$s\in [0,\varepsilon(x_1)]$ we could estimate $I_s^t(r)$ as:
\begin{align*}
I_{s}^t(r)\ 
&\le\ \int_0^{s}\left(
\frac{1}{\mu_f(\mathcal{B}_{u}^t(\alpha, r))^2}
\int_{\psi_u(\mathcal{B}_u^t(r))\times \psi_u(\mathcal{B}_u^t(r))}
\left(\int_{\gamma_{x,y}}|Hess_{h_{cr^2}}|\right)\ \dmu_f^2\right)\ \text{d}u\\
&\le\
\int_0^{s}\left(10r\ C_{Seg}
\frac{\mu_f(\psi_u(\mathcal{B}_u^t(\alpha,
r)))}{\mu_f(\mathcal{B}_u^t(\alpha, r))^2}
\int_{B(\gamma(t-s),5r)}|Hess_{h_{cr^2}}|\ \dmu_f\right)\ \text{d}u\\
&\le\ \int_0^{s}\left(10r\ C_{Seg}
\frac{\mu_f(B(\gamma(t-u),2r))}{\mu_f(\mathcal{B}_u^t(\alpha, r))^2}
\int_{B(\gamma(t-s),5r)}|Hess_{h_{cr^2}}|\ \dmu_f\right)\ \text{d}u.
\end{align*}
Moreover, by the volume doubling property (\ref{eqn: volume_doubling}) within
$B(p_0,D)$, assumption (\ref{eqn: volume_estimate}) and Lemma \ref{L303}, we
could continue to estimate: $\forall s\in [0,\varepsilon(x_1)],$
\begin{align}
\begin{split}
 I_s^t(r)\ &\le\ 
\int_0^{s}\left(10r\ C_2C_{Seg} \left(
\frac{\mu_f(B(\gamma(t-u),r))}{\mu_f(\mathcal{B}_u^t(\alpha, r))}\right)^2
\aint_{B(\gamma(t-u),5r)}|Hess_{h_{cr^2}}|\ \dmu_f\right)\ \text{d}u\\
&\le\ 
10r\ C_2C_{Seg}\left(\int_{\delta l}^{l-\delta l}
\aint_{B(\gamma(u),5r)}|Hess_{h_{cr^2}}|^2\
\dmu_f\right)^{\frac{1}{2}}\sqrt{s}\\
&\le\ 10 C_2C_{Seg}\sqrt{C_3}\ r\sqrt{s\slash l},
\end{split}
\label{eqn: distance_3}
\end{align}
Now (\ref{eqn: distance_1}), (\ref{eqn: distance_2}) and
(\ref{eqn: distance_3}) together imply that for almost every $x_1\in
T_{\eta}^r$ and $x_2\in T_{\eta}^r(x_1)\cap \mathcal{B}_s^t(\alpha, r),$
\begin{align}
\forall s\in [0,\varepsilon(x_1)],\quad |d(\psi_s(x_1),
\psi_s(x_2))-d(x_1,x_2)|\ \le\ C_6\eta^{-2}r \sqrt{s\slash l},
\label{eqn: distance_2}
\end{align}
where $C_6= C_6(m,F,K,D,\delta):= 4C_4+10C_2C_{Seg}\sqrt{C_3}$. Here we
emphasize that in proving this estimate we only need $x_1\in T_{\eta}^r\cap
\mathcal{B}_s^t(\alpha, r)$ and $x_2\in T_{\eta}^r(x_1)\cap \mathcal{B}_s^t(r)$.
The stronger assumption that $x_2 \in \mathcal{B}_s^t(\alpha, \xi r)$ is not
needed yet.

Now we put $\varepsilon_0:=(1+\alpha)^2 \eta^4\slash (16C_6^2)$. Suppose, for
the purpose of a contradiction argument, that $\varepsilon(x_1)<
\varepsilon l\le \varepsilon_0 l$, then since $x_2\in
\mathcal{B}_s^t(\alpha, \xi r)$, we have $d(\psi_s(x_2),\gamma(t-s))\le
(1+\alpha)\xi r$ whenever $s\in [0,\varepsilon(x_1)]$, and the
triangle inequality implies that
\begin{align}\label{eqn: distance_3}
d(\psi_{\varepsilon(x_1)}(x_1),\gamma(t-\varepsilon(x_1)))\ \le\
(\frac{1}{4}+\xi)(1+\alpha)r,
\end{align}
whence a desired contradiction to (\ref{eqn: away}).

Therefore $\varepsilon(x_1)=\varepsilon l$, and (\ref{eqn: distance_2}) is
valid for all $s\in [0,\varepsilon l]$. Especially, this is (\ref{eqn:
distance_estimate}) holding for all $s\in[0,\varepsilon l]$, as claimed by the
lemma.
\end{proof}

\begin{remark}
Let us emphasis that the estimate (\ref{eqn: distance_estimate}) depends on
(\ref{eqn: volume_estimate}) whose range of validity depends
on the specific manifold, geodesic and scale $r$. But with these estimates we
are now ready to remove such dependence of $\varepsilon$ in (\ref{eqn:
volume_estimate}).
\end{remark}

On the other hand, we could consider the gradient flow $\psi_s'$ 
generated by the almost everywhere defined vector field $-\nabla d_{q'}$, with
$d_{q'}:M\to [0,\infty)$ denoting the geodesic distance to $q'\in M$. We could
define $\mathcal{B}_s^t(\alpha, r)'$ as following
\begin{align*}
\mathcal{B}_s^t(\alpha, r)'\ :=\ \{z\in B(\gamma(t-s),r):\ \forall u\in [0,s
l],\ \psi'_u(z)\in B(\gamma(t+u),(1+\alpha)r)\}.
\end{align*}
By the symmetry of $\gamma$, we could apply Lemma~\ref{lem: distance_estimate}
to see that as long as 
\begin{align}\label{eqn: volume_estimate'}
\forall s\in [0,\varepsilon l],\quad
\frac{\mu_f(\mathcal{B}_s^t(\alpha,r)')}{\mu_f(B(\gamma(t),r))}\ \ge\
\frac{1}{2},
\end{align}
then $\forall x_1\in (T_{\eta}^r)'$ and $\forall x_2\in T_{\eta}^r(x_1)'\cap
\mathcal{B}_s^t(\alpha,\xi r)'$, 
\begin{align}\label{eqn: distance_estimate'}
\forall s\in [0,\varepsilon l],\quad |d(\psi'_s(x_1),\psi'_2(x_2))-d(x_1,x_2)|\
\le\ C_6\eta^{-2}r\sqrt{s\slash l},
\end{align}
where $(T_{\eta}^r)',T_{\eta}^r(x_1)'\subset B(\gamma(t-s),r)$ are defined in
the same way as $T_{\eta}^r$ and $T_{\eta}^r(x_1)$ in (\ref{eqn: Tr}) and
(\ref{eqn: Trx}) respectively, but with $h'_{cr^2}$ --- the parabolic smoothing
of $-d_{q'}$ --- replacing $h_{cr^2}$ in (\ref{eqn: Fxy}) and (\ref{eqn:
int_I}).
\begin{lem}
There exists a small $\varepsilon_1=\varepsilon_1(\eta\ |m,F,K,D,\delta)>0$
such that fixing $r\in (0,\delta l \slash 10)$, the estimates 
(\ref{eqn: volume_estimate}) and (\ref{eqn: volume_estimate'}) hold for some
$\varepsilon \ge \varepsilon_1$.
In fact, we have
$(T_{\eta}^r\backslash (\mathcal{C}_q\cup \mathcal{C}_{q'}))\subset
\mathcal{B}_s^t(\alpha,r)$ and $((T_{\eta}^r)'\backslash
(\mathcal{C}_q\cup \mathcal{C}_{q'}))\subset \mathcal{B}_s^t(\alpha,r)'$
whenever $\forall s\in [0,\varepsilon_1 l]$.
\label{lem: big_A}
\end{lem}

\begin{proof}
Fix $r\in (0,\delta l \slash 10)$, and let $\varepsilon$ be the largest possible
such number that both (\ref{eqn: volume_estimate}) and (\ref{eqn:
volume_estimate'}) hold. Again, this $\varepsilon$ is positive but its value
depends on the specific $M$ and $\gamma$. We will choose an $\varepsilon_1$
depending only on $m,F,K,D,\delta$ and $\eta$ such that were
$\varepsilon<\varepsilon_1$ to hold then a contradiction will be deduced.

\emph{\textbf{Step 1: Connecting to the good core neighborhood.}}
Recall that Lemma~\ref{lem: core_nbhd} tells that there are a small
$r'=r'(M,\gamma)>0$ and a core neighborhood $H_{r'}^t= B(\gamma(t),r')$, such
that it stays close to $\gamma$ under the geodesic flow. Let us now fix this
neighborhood of $\gamma(t)$, which depends on specific $M$ and $\gamma$. Notice
that if we set $\varepsilon_1':=(\ln (1+\alpha)\slash C_5)^2\delta $, then by
the definition of $H_{r'}^t$ and the proof of Lemma~\ref{lem: core_nbhd}, we have
\begin{align}
\forall s\in [0,\varepsilon_1'l],\ \forall x\in H_{r'}^t,\quad
d(\psi_s(x),\gamma(t+s))\ \le\ (1+\alpha)d(x,\gamma(t)).
\label{eqn: core_inclusion}
\end{align} 

We also let $\xi$ be some small positive number, say $\xi = \frac{1}{20}$, and
 let $r_i:=\xi^{i}r$ for $i=0,1,2,\ldots,I$, where
 $I:=\left\lceil \log_{\xi}\frac{r'}{2r}\right\rceil$ is defined to be the first
 natural number such that $r_I\le r'\slash 2$.

Now for an arbitrary $x_0\in T_{\eta}^r\backslash (\mathcal{C}_q\cup
\mathcal{C}_{q'})$ fixed, our plan is to connect it to $H_{r'}^t$ by selecting
$\{x_i\}_{i=0}^{I}$ inductively: suppose $x_i$ is chosen, then pick any
$x_{i+i}\in (T_{\eta}^{r_i}(x_i)\cap T_{\eta}^{r_{i+1}})\backslash
\mathcal{C}_q$. This is doable because (\ref{eqn: Chebyshev}) is independent of
$r$: as long as we choose $\eta\le (C(m,F,K,D)\xi\slash 4)^m$ with $C(m,F,K,D)$
coming from the volume comarison (Theorem~\ref{T101}), then we have
\begin{align*}
\mu_f(T_{\eta}^{r_i}(x_i))+\mu_f(T_{\eta}^{r_{i+1}})\ &\ge\
(1-2\eta)\left(\mu_f(B(\gamma(t),r_i))+\mu_f(B(\gamma(t),r_{i+1}))\right)\\
&\ge\ (1-2\eta)(1+C(m,F,K,D)^m\xi^m)\mu_f(B(\gamma(t),r_i))\\
&>\ \mu_f(B(\gamma(t),r_i)),
\end{align*}
i.e. $T_{\eta}^{r_i}(x_i)\cap T_{\eta}^{r_{i+1}}$ has positive weighted volume,
and especially a non-empty intersection outside the cut-locus of $q$. We denote
by $\sigma_i$ the integral curve of $-\nabla d_q$ with initial value $x_i$.
Each $\sigma_i$ is a minimal geodesics.

\emph{\textbf{Step 2: Estimating the distance.}}
According to (\ref{eqn: core_inclusion}), $x_I\in \mathcal{B}_s^t(\alpha,r_I)$
whenever $s\in [0,\varepsilon_1'l]$. Therefore, applying Lemma~\ref{lem:
distance_estimate} to $x_{I-1}\in T_{\eta}^{r_{I-1}}$ and $x_{I}\in
T_{\eta}^{r_{I-1}}(x_{I-1})\cap \mathcal{B}_s^t(\alpha,r_I)$, by (\ref{eqn:
distance_3}) we have
\begin{align*}
\forall s\in [0,\min\{\varepsilon l,\varepsilon_0 l,\varepsilon_1'l\}],\quad
d(\psi_s(x_{I}),\psi_s(x_{I-1}))\ \le\
\left(\frac{1}{4}+\xi\right)(1+\alpha)r_{I-1}.
\end{align*}
This further implies that for any $s\le \min\{\varepsilon
l,\varepsilon_0 l,\varepsilon_1'l\}$,
\begin{align}
\begin{split}
d(\psi_s(x_{I-1}),\gamma(t-s))\ &\le\
d(\psi_s(x_I),\gamma(t-s))+\left(\frac{5}{4}+\xi\right)r_{I-1}\\
&\le\ \left(\frac{1}{4}+2\xi\right)(1+\alpha)r_{I-1}.
\end{split}
\label{eqn: distance_I}
\end{align}
Especially, $x_{I-1}\in \mathcal{B}_s^t(\alpha,r_{I-1})$ whenever $s\le
\min\{\varepsilon l,\varepsilon_0 l,\varepsilon_1'l\}$.

We could now apply Lemma~\ref{lem: distance_estimate} to the pair of points
$x_{I-2}$ and $x_{I-1}$, and conclude that $x_{I-2}\in
\mathcal{B}_s^t(\alpha, r_{I-2})$ whenever $s\le \min\{\varepsilon
l,\varepsilon_0 l,\varepsilon_1'l\}$.
Repeating the same argument another $I-2$ steps, we will get for any $s\le
\min\{\varepsilon l,\varepsilon_0 l,\varepsilon_1'l\}$,
\begin{align*}
d(\psi_s(x_0),\gamma(t-s))\ \le\
&d(\psi_s(x_I),\gamma(t-s))+\left(\frac{1}{4}+\xi\right)(1+\alpha)r
\sum_{i=0}^{I-1}\xi^i\\
<\ &(1+\alpha)r,
\end{align*}
by the choice of $\xi$. Especially, this implies that $x_0\in
\mathcal{B}_s^t(\alpha, r)$ whenever $s\le \min\{\varepsilon l,\varepsilon_0
l,\varepsilon_1'l\}$. By (\ref{eqn: distance_estimate'}) and the same reasoning,
we see that $(T_{\eta}^r)'\backslash (\mathcal{C}_q\cup \mathcal{C}_{q'})\subset
\mathcal{B}_s^t(\alpha, r)'$ whenever $s\le \min\{\varepsilon l,\varepsilon_0
l,\varepsilon_1'l\}$.

\emph{\textbf{Step 3: Bounding $\varepsilon$ from below.}}
Let $\varepsilon''\in (0,\delta\slash 10)$ be the largest constant
satisfying
\begin{align*}
\forall s\in [0,\varepsilon''l],\ \forall \delta l\le a< b\le
(1-\delta)l,\quad
\int_{a+s}^{b+s}e^{F(2D)u}\mathcal{A}_K^{m-1}(u)\ \text{d}u\ \le\ \frac{10}{9}
\int_{a}^{b}e^{F(2D)u}\mathcal{A}_K^{m-1}(u)\ \text{d}u,
\end{align*}
then clearly $\varepsilon''$ is determined by $m$, $F$, $K$, $D$ and $\delta$. 
Setting $\varepsilon_1:=\min\{\varepsilon_0,\varepsilon_1',\varepsilon''\}$, we
show that whenever $\eta\le 100^{-m}$, $\varepsilon\ge \varepsilon_1$ by a
contradiction argument:

Otherwise, notice that $\mu_f(\mathcal{B}_{s}^t(\alpha, r))\slash
\mu_f(B(\gamma(t-s),r))$ varies continuously with respect to $s$, then by
(\ref{eqn: volume_estimate}), (\ref{eqn: volume_estimate'}) and the maximality
of $\varepsilon$, we have 
\begin{align}\label{eqn: contradiction_main}
\frac{\mu_f(\mathcal{B}_{\varepsilon
l}^t(\alpha, r))}{\mu_f(B(\gamma(t-\varepsilon l),r))} = \frac{1}{2}\quad
\text{or}\quad \frac{\mu_f(\mathcal{B}_{\varepsilon
l}^t(\alpha, r)')}{\mu_f(B(\gamma(t),r))} = \frac{1}{2}.
\end{align}  
Now suppose it is the first case. Since $\varepsilon <\varepsilon_1$, we have
$(T_{\eta}^r\backslash (\mathcal{C}_q\cup \mathcal{C}_{q'}))\subset
\mathcal{B}_{\varepsilon l}^t(\alpha,r)$, and thus 
\begin{align*}
\frac{\mu_f(\mathcal{B}_{\varepsilon l}^t(\alpha,
r))}{\mu_f(B(\gamma(t-\varepsilon l),r))}\ &\ge\ (1-\eta)
\frac{\mu_f(T_{\eta}^r)}{\mu_f((T_{\eta}^r)')}\\
&\ge\
\frac{9(1-\eta)^2}{10}
\frac{\mu_f(B(\gamma(t),r)}{\mu_f(\psi'_s((T_{\eta}^r)'))}\\
&\ge\ \frac{9(1-\eta)^2}{10}
\frac{\mu_f(B(\gamma(t),r)}{\mu_f(B(\gamma(t),(1+\alpha)r))}\\
&\ge\ \frac{9(1-\eta)^2}{10(1+C(m,F,K,D)\alpha)^m},
\end{align*}
where $C(m,F,K,D)>0$ is a constant determined by the volume comparison in
Theorem~\ref{T101}. Here the second inequality holds thanks to the comaprison
(\ref{eqn: area_comparison_translated}) --- especially, notice that we are
following the gradient flow $\psi_s'$ of $-\nabla d_{q'}$. Compare the proof of
the segment inequality and $\alpha\in (0,1)$ can be chosen (depending on
$C(m,F,K,D)$) so that $(1-\eta)^2(1+C(m,F,K,D)\alpha)^{-m} \ge\frac{5}{6}$, and
this leads to a contradiction. Similar contradiction could also be deduced if
the second case of (\ref{eqn: contradiction_main}) were assumed. Therefore
$\varepsilon\ge \varepsilon_1$, a constant \emph{solely} determined by
$m,F,K,D,\delta$ and $\eta$. We notice here that $\lim_{\eta\to 0}
\varepsilon_1(\eta|m,F,D,\delta)=0$ by the definition of $\varepsilon_1$.
\end{proof}

We are now ready to estimate the Gromov-Hausdorff distance of metric balls of
arbitrarily small size $r$:
\begin{thm}[Gromov-Hausdorff distance between nearby metric balls]
\label{thm: Holder_continuity}
Fix a space $(M,p,g,f)$ in the moduli $\mathcal{N}_m(F,K)$, then for any
$\delta\in (0,1\slash 10)$ and $\epsilon \in (0,\delta\slash 10)$ fixed, there are
constants $C_7,C_8$ only depending on $m,F,K,D,\delta$ such that on any minimal
geodesic $\gamma$ contained in $B(p,D)$ with $|\gamma|=l$, and for any $x,y$ on
$\gamma([\delta l,(1-\delta)l])$,
\begin{align*}
\frac{d(x,y)}{l}\ \le\ C_7\quad \Rightarrow\quad
d_{GH}(B(x,r),B(y,r))\ \le\
C_8\left(\frac{d(x,y)}{l}\right)^{\frac{1}{4m+2}}r.
\end{align*}
\end{thm}

\begin{proof}
Since the estimate is symmetric in terms of $x$ and $y$, we only argue in one
direction. Now fix any $s\in [0,\varepsilon_1l]$. We immediately have
$\mu_f(\mathcal{B}_s^t(\alpha, r))\ge (1-2\eta)\mu_f(B(\gamma(t),r))$, in view
of Lemma~\ref{lem: big_A} and (\ref{eqn: Chebyshev}). Moreover, we have
$\mu_f(\mathcal{B}_s^t(\alpha, r)\cap T_{\eta}^r(x))\ge
(1-4\eta)\mu_f(B(\gamma(t),r))$ for any $x\in T_{\eta}^r$, still because of
(\ref{eqn: Chebyshev}). Such volume estimates, together with the volume
comparison within $B(p_0,D)$, imply that $T_{\eta}^r$ and $T_{\eta}^r(x)\cap
\mathcal{B}_s^t(r)$ are $4C_2^{-\frac{1}{m}}\eta^{\frac{1}{m}}r$-dense subsets
of $B(\gamma(t),r)$, whenever $x\in T_{\eta}^r$.

Recall that (\ref{eqn: distance_2}) tells that for any
$x_1\in T_{\eta}^r$ and any $x_2\in T_{\eta}^r(x_1)\cap
\mathcal{B}_s^t(\alpha, r)$,
\begin{align*}
|d(\psi_s(x_1),\psi_s(x_2))-d(x_1,x_2)|\ \le\ C_5\eta^{-2}r
\sqrt{s\slash l}.
\end{align*}
Now for any $x_1,x_2\in T_{\eta}^r$, since $\mu_f(T_{\eta}^r(x_1)\cap
T_{\eta}^r(x_2)\cap \mathcal{B}_s^t(\alpha, r))\ge
(1-8\eta)\mu_f(B(\gamma(t),r))$, we could select some 
\begin{align*}
y\in T_{\eta}^r(x_1)\cap T_{\eta}^r(x_2)\cap \mathcal{B}_s^t(\alpha, r)\cap
B(x_1,8C_2^{-\frac{1}{m}}\eta^{\frac{1}{m}}r),
\end{align*}
and estimate
\begin{align*}
|d(\psi_s(x_1),\psi_s(x_2))-d(x_1,x_2)|\ &\le\
|d(\psi_s(x_2),\psi_s(y))-d(x_2,y)|+d(\psi_s(x_1),\psi_s(y))+d(x_1,y) \\
&\le\ 2C_5\eta^{-2}r\sqrt{s\slash l}+2d(x_1,y)\\
&\le\ 2r(C_5\eta^{-2}\sqrt{s\slash l}+8C_2^{-\frac{1}{m}}\eta^{\frac{1}{m}}).
\end{align*}
 In order to make this last
line of the last estimate having only $s$ as the variable, we would like to
choose $\eta$ according to the value of $s$. Notice that in order for this
estimate to hold, we cannot violate $s\le \varepsilon_1 l$; on the other hand,
let us recall that $\varepsilon_1$ depends on $\eta$, which ultimately comes into
play via $\varepsilon_0=\eta^4\slash (16C_6^2)$. Therefore, we will first
choose $\eta(s)$, then check that $s\slash l\le \eta(s)^4\slash (16C_6^2)$ and
$\eta(s)\le 10^{-2m}$ (see Step 1 of Lemma~\ref{lem: distance_estimate}):
let
\begin{align*}
\eta(s)\ :=\ C_6^{\frac{m}{2m+1}}C_2^{\frac{1}{2m+1}}(s\slash
(8l))^{\frac{m}{4m+2}},
\end{align*}
then 
\begin{align}
|d(\psi_s(x_1),\psi_s(x_2))-d(x_1,x_2)|\ \le\ 4(8^mC_6\slash
C_2^2)^{\frac{1}{2m+1}}(s\slash l)^{\frac{1}{4m+2}}r.
\label{eqn: GH_distance_1}
\end{align}
Moreover, the requirements that $s\slash l\le \eta(s)\slash (16C_6^2)$ and
$\eta(s)\le 10^{-2m}$ translate as
\begin{align*}
\frac{s}{l}\ \le\
8C_6^{-2}\min\left\{C_0^42^{-7-14m},C_2^{-\frac{2}{m}}10^{-4-8m}\right\}\ =:\
C_7(m,F,K,D,\delta),
\end{align*}
the right-hand side of which, being a constant only depending on $m, D$,
and $\delta$.
Therefore, once $s\slash l$ is below this $C_7$, the previous requirements of
$\eta(s)$ are met, and all previous estimates go through with no problem. 

Therefore, whenever $d(x,y)\le C_7(m,D,\delta)l$, we have, by the density
estimate and (\ref{eqn: GH_distance_1}), the desired estimate, with the constant 
$C_8(m,F,K,D,\delta)=24(8^mC_6\slash C_2^2)^{\frac{1}{2m+1}}$.
\end{proof}

\subsection{Extension of limit minimal geodesics}
Throughout this subsection, we fix a sequence $\{(M_i,p_i,g_i,f_i,)\}\subset
\mathcal{N}_m(F,K)$ that converges in the pointed-Gromov-Hausdorff topology to a
limit $(X,p_{\infty},d_{\infty},f_{\infty})$. Our focus will be on the
non-compact case, which is more complicated and natural to consider. We will
provide detailed proofs for the necessary adjustments to generalize
Colding-Naber's argument in~\cite[Sections 1.2 and 1.4]{CN12} to
$\mathcal{N}_m(F,K)$ limits.

In view of Theorem~\ref{thm: Holder_continuity}, we could only compare geodesic
balls centered at two points that are away from the endpoints of a minimal
geodesic connecting them. For a fixed complete Riemannian manifold $(M,g)$, and
any pair of points $x,y\in M$, we let $\gamma_{xy}$ denote a minimal geodesic
connecting them. Due to the possible existence of the cut-locus, not every pair
of points $(x,y)\in M\times M$ sees their $\gamma_{xy}$ minimally extensible
to both ends. But the minimal extensibility holds for \emph{almost every} pair
of points, with respect to the natural product measure on $M\times M$.

In order to prove the almost everywhere extensibility of limit minimal geodesics
on a pointed-Gromov-Hausdorff limit, we have to show that the problematic
cut-loci do not accumulate to acquire positive limit measure during the
convergence. The key observation, due to Shouhei Honda~\cite{H11}, is that
the cut-loci could be characterized by an inequality --- the non-vanishing of the
excess function~\cite{AG90} --- whose effective version persists to the
Gromov-Hausdorff limit. The proof of what we need for $\mathcal{N}_m(F,K)$
limits is the same as the original one in~\cite[Appendix A]{CN12}, except at
one point: the crucial estimate in~\cite[Lemma A.2]{CN12} relies on the
Laplacian comparison with a uniform Ricci curvature lower bound, which is not
available for manifolds in $\mathcal{N}_m(F,K)$. The following lemma fills this
only gap in carrying Colding-Naber's original argument to our setting:
\begin{lem}
Suppose $(M,p,g,f)\in \mathcal{N}_m(F,K)$. For each $\delta,r\in (0,1)$, $D>0$
and $k\in \mathbb{N}$, there exists a constant $C_9=C_9(m,F,K,D,\delta)$
such that
\begin{align}
\frac{\mu_{f\cdot f}\left(\mathcal{C}_M(r,k)\cap
A_{\delta,\delta^{-1}}(D)\right)}{\mu_f(B(p,1))^2}\ \le\ C_9\ r.
\label{eqn: effective_cut_loci}
\end{align}
\end{lem}

 Here we define, as~\cite[(A.5)]{CN12},
\begin{align*}
\mathcal{C}_M(r,k)\ :=\ \left\{(x,y)\in M\times M:\ \forall z,w\in M,\
d^2(x,z)+d^2(y,w)\ge 2r^2\Rightarrow e_{(z,w)}(x,y)\ge k^{-2}\right\},
\end{align*}
and 
\begin{align*}
e_{(z,w)}(x,y)\ :=\
2^{-\frac{1}{2}}d(x,y)+(d(x,z)^2+d(y,w)^2)^{\frac{1}{2}}-2^{-\frac{1}{2}}d(z,w),
\end{align*}
is the excess function on the isometric product manifold $(M\times M, g\oplus
g)$, see~\cite[(A.2)]{CN12}. Notice that $\mathcal{C}_M(r,k)$ is just a
quantitative version of the $r$-cut-loci $\mathcal{C}_M(r)$ of the manifold
$(M\times M, g\oplus g)$, which is actually
$\cup_{k=1}^{\infty}\mathcal{C}_M(r,k)$. Moreover, $(x,y)\in \mathcal{C}_M(r)$
if and only if the geodesics emanating from the midpoint towards $x$ and $y$
respectively are minimal till they reach $x$ and $y$, but at least one of them
could not be extended beyond $x$ or $y$ as a minimal geodesic for at least a
distance of $r$.

Furthermore, the set $A_{\delta,\delta^{-1}}(D)$ is defined for any $\delta >0$
as
\begin{align*}
A_{\delta,\delta^{-1}}(D)\ :=\ \left\{(x,y)\in M\times M:\
p_{xy}\in B(p,D),\ \delta\le d_{\mathbf{\Delta}}(x,y)\le
\delta^{-1}\right\},
\end{align*}
with $p_{xy}$ denoting the midpoint of a minimal geodesic
connecting $x$ and $y$, and $d_{\mathbf{\Delta}}(x,y)=2^{-\frac{1}{2}}d(x,y)$
denoting the distance between $(x,y)$ and the diagonal $\mathbf{\Delta}$ of
$M\times M$, in the product metric.

We also notice that the product of $(M,p,g,f)$ by itself is
$\left(M\times M,(p,p), g\oplus g, f\cdot f\right)$, where
$\forall x,y\in M, f\cdot f(x,y):=f(x)f(y)$. We have the product an element of
$\mathcal{N}_{2m}(\bar{F},K)$, where
$\bar{F}(x,y):=\sqrt{F(d(p,x))^2+F(d(p,y))^2}$, since clearly we have
\begin{align*}
\forall x,y\in M,\quad |\nabla f\cdot f|^2(x,y)\ &\le\
F(d(p,x))^2+F(d(p,y))^2,\\
\text{and}\quad Rc_{g\oplus g}+\nabla^2(f\cdot f)\ &\ge\ -Kg\oplus g.
\end{align*}
Due to the metric product structure, the distance to the diagonal
$d_{\mathbf{\Delta}}$ enjoys the following Laplace comparison inequality:
$\forall x,y\in B(p,D)$,
\begin{align}
\begin{split}
\Delta_{f\cdot f}d_{\mathbf{\Delta}}(x,y)\ =\
&\frac{1}{\sqrt{2}}(\Delta_f)_xd(x,y)+\frac{1}{\sqrt{2}}(\Delta_f)_yd(x,y)\\
\le\ &\sqrt{2}\left(\frac{m-1}{d(x,y)}+(m-1)\sqrt{K}+F(2R)\right).
\end{split}
\end{align}
 By~\cite[Lemma A.1]{CN12}, the distance of $(x,y)$ to the diagonal $D$ is
 realized by the distance to $(p_{xy},p_{xy})$, the midpoint of the minimal
 geodesic connecting $x$ and $y$. The geodesic ray realizing the distance to the
 diagonal is then given as:
\begin{align*}
\forall (x,y)\in M\times M,\ s\mapsto
\exp_{(p_{xy},p_{xy})}s(\mathbf{v}_{xy},-\mathbf{v}_{xy}),
\end{align*}
where $\mathbf{v}_{xy}:=\dot{\gamma}_{xy}(\frac{1}{2}d(x,y))$, and
$(\mathbf{v}_{xy},-\mathbf{v}_{xy})\in S_{p_{xy}}M\times S_{p_{xy}}M$, the
product of unit tangent vectors.
Associated to this exponential map, we could consider
\begin{align*}
T_r(D)\ :=\ &\left\{(x,y)\in M\times M:\ d_{\mathbf{\Delta}}(x,y)\le r\
\text{and some}\ p_{xy}\in B(p,D)\right\}\\
=\ &\left\{\exp_{(q,q)}s(\mathbf{v},-\mathbf{v}):\ s\in [0,r\slash \sqrt{2}),\
p\in B(p,D),\ \mathbf{v}\in S_qM\right\},
\end{align*}
which is the open $r$-tubular neighborhood of the diagonal $\mathbf{\Delta}
\subset B(p_0,D)\times B(p_0,D)$. Clearly,
$A_{\delta,\delta^{-1}}(D)=T_{\delta^{-1}}(D)\backslash T_{\delta}(D)$, and 
\begin{align*}
 \partial T_s(D)\ =\ &\left\{(x,y)\in M\times M:\ d_{\mathbf{\Delta}}(x,y)=s\
 \text{and some}\ p_{xy}\in B(p,D)\right\}\\
=\ &\left\{\exp_{(q,q)}s(\mathbf{v},-\mathbf{v}):\ (q,\mathbf{v})\in
SB(p,D)\right\},
\end{align*}
where $SB(p,D)$ is the sphere bundle of $B(p,D)$. 

We denote the area form of $\partial T_{s}(D)$ at
$\exp_{(q,q)}s(\mathbf{v},-\mathbf{v})$ by $\mathcal{A}^2(q,\mathbf{v},s)$, and
the $e^{-f\cdot f}$-weighted product measure density on $\partial T_{s}(D)$ by
$\mathcal{A}^2_{f\cdot f}(q,\mathbf{v},s):=e^{-f\cdot
f}\mathcal{A}^2(q,\mathbf{v},s)$.
We notice that
\begin{align*}
\partial_s\ln \left(\mathcal{A}^2_{f\cdot f}(q,\mathbf{v},s)\right)\ =\
&\Delta_{f\cdot f}d_{\mathbf{\Delta}}(\exp_qs\mathbf{v},\exp_q-s\mathbf{v})\\
\le\ &\sqrt{2}\left((m-1)s^{-1}+(m-1)\sqrt{K}+F(D+\sqrt{2}s)\right),
 \end{align*} 
therefore, for any fixed $(q,\mathbf{v})\in SB(p,D)$ and $s\in
[\delta,\delta^{-1}]$, the ratio
\begin{align*}
\frac{\mathcal{A}^2_{f\cdot f}(q,\mathbf{v},s)}{s^{\sqrt{2}(m-1)}
e^{((m-1)\sqrt{2K}+\sqrt{2}F(D+\sqrt{2}\delta^{-1}))s}}
\end{align*}
is monotone non-increasing with respect to $s$.

Recall that $(x,y)\in \mathcal{C}_M(r)$ if and only if the geodesics
$s\mapsto \exp_{p_{xy}}s\mathbf{v}$ and $s\mapsto \exp_{p_{xy}}-s\mathbf{v}$ are
minimal for $s\in (0,\frac{1}{2}d(x,y))$, but cease to be so for some $s\in
(\frac{1}{2}d(x,y),\frac{1}{2}d(x,y)+r)$. Therefore 
\begin{align*}
\forall \delta>0,\ \forall (q,\mathbf{v})\in SB(p,D),\quad \left|\left\{s\in
(\delta,\delta^{-1}):\ \exp_{q}s(\mathbf{v},-\mathbf{v})\in
\mathcal{C}_M(r)\right\}\right|\ \le\ r.
\end{align*}
We now put the estimates together to see:
\begin{align*}
\mu_{f\cdot f}\left(\mathcal{C}_M(r)\cap
A_{\delta,\delta^{-1}}(D)\right)\ =\
&\int_{\delta}^{\delta^{-1}}\left(\int_{\partial
T_s(D)}\chi_{\mathcal{C}_M(r)}e^{-f\cdot f}\right)\
\text{d}s\\
=\ &\int_{SB(p,D)}\int_{\delta}^{\delta^{-1}}
\chi_{\mathcal{C}_M(r)}\mathcal{A}^2_{f\cdot f}(q,\mathbf{v},s)\
\text{d}s\ \text{d}\sigma(q,\mathbf{v})\\
\le\
&\int_{SB(p,D)}
\left(\int_{\delta}^{\delta^{-1}}\chi_{\mathcal{C}_M(r)}\
\text{d}s\right)
\frac{\mathcal{A}^2_{f\cdot f}
(q,\mathbf{v},\delta)}{\delta^{\sqrt{2}(m-1)}e^{((m-1)\sqrt{2K}
+\sqrt{2}F(D+\sqrt{2}\delta^{-1}))\delta}}\ \text{d}\sigma(q,\mathbf{v})\\
\le\ &\int_{SB(p,D)}r\
\frac{\mathcal{A}^2_{f\cdot f}
(q,\mathbf{v},\delta)}{\delta^{\sqrt{2}(m-1)}e^{((m-1)\sqrt{2K}
+\sqrt{2}F(D+\sqrt{2}\delta^{-1}))\delta}}\ \text{d}\sigma(q,\mathbf{v})\\
=\ &\frac{ \mu_{f\cdot f}(\partial
T_{\delta}(D))}{\delta^{\sqrt{2}(m-1)}
e^{((m-1)\sqrt{2K}+\sqrt{2}F(D+\sqrt{2}\delta^{-1}))\delta}}\ r.
\end{align*}
Finally, since $\forall x\in B(p,D+1)$ fixed $\{y\in B(p,D+1): (x,y)\in \partial
T_{\delta}(D)\}=\partial B(x,\sqrt{2}\delta)\cap B(p,D+1)$, we have, by
Theorem~\ref{T101},
\begin{align*}
\mu_{f\cdot f}(\partial T_{\delta}(D))\ \le\
&\int_{B(p,D+\sqrt{2}\delta)}\mu_f(\partial B(x,\sqrt{2}\delta))\
e^{-f(x)}\dvol_g(x)\\
\le\
&\frac{e^{F(D+\sqrt{2}\delta^{-1})(D+\delta^{-1})}
Area_K^{m-1}(\sqrt{2}\delta)}{Vol_K^m(\sqrt{2}\delta)}
\int_{B(p,D+\sqrt{2}\delta)} \mu_f(B(x,\sqrt{2}\delta))\ e^{-f(x)} \dvol_g(x),
\end{align*}
and applying Theorem~\ref{T101} again we have
\begin{align*}
\mu_{f\cdot f}\left(\mathcal{C}_M(r)\cap
A_{\delta,\delta^{-1}}(D)\right)\ \le\ C_9\ \mu_f(B(p,1))^2 r,
\end{align*}
where
\begin{align*}
C_9\ =\ C_9(m,F,K,D,\delta)\ :=\
\frac{e^{F(D+\sqrt{2}\delta^{-1})(D+\delta^{-1})}}{\delta^{\sqrt{2}(m-1)}
e^{((m-1)\sqrt{2K}+\sqrt{2}F(D+\sqrt{2}\delta^{-1})\delta)}}
\frac{Area_K^{m-1}(\sqrt{2}\delta)}{Vol_K^m(\sqrt{2}\delta)}
\left(\frac{Vol_K^m(D+\sqrt{2})}{Vol_K^m(1)}\right)^2.
\end{align*}
Therefore we get the desired estimate (\ref{eqn: effective_cut_loci}), since
$\mathcal{C}_M(r)=\cup_{k=0}^{\infty}\mathcal{C}_M(r,k)$.

The rest of the argument in showing the almost everywhere extensibility follows
verbatim as the rest of~\cite[Appendix A]{CN12}, as well as~\cite{H11}. So we
have shown that with respect to the limit measure, almost every pair of points
lie in a minimal geodesic that minimally extends to both ends:
\begin{lem}[Extension of limit minimal geodesics]
Let a sequence  $\{(M_i,p_i,g_i,f_i)\}\subset \mathcal{N}_m(F,K)$
converge to $(X,p_{\infty},d_{\infty},f_{\infty})$ in the
pointed-Gromov-Hausdorff topology, such that their associated renormalized
measures $\nu_{f_i}$ also converge to a limit measure $\nu_{\infty}$ on $X$,
then $\nu_{\infty}\times \nu_{\infty}$ almost every pair of point $(x,y)\in
X\times X$ lies in the interior of some limit minimal geodesic.
\label{lem: extension}
\end{lem} 

Now we are in a position to state and prove the (weak) convexity of the regular
part in the Gromov-Hausdorff limit of a sequence of Ricci shrinkers. To start
the discussion, let us state an immediate consequence of the H\"older
continuity of the geodesic balls along a geodesic segment:
\begin{prop}[H\"older continuity of tangent cones]
\label{prop: tangent_cone_continuity} 
Let $(X,p_{\infty},d_{\infty},f_{\infty})$ be a pointed-Gromov-Hausdorff limit
of a sequence in $\mathcal{N}_m(F,K)$, then the tangent cones resulted from the same
scaling sequence varies H\"older-continuously in the Gromov-Hausdorff topology,
as the point varies in the interior of limit minimal geodesics.
\end{prop}
\begin{proof}
Let $\gamma_{\infty}:[0,l]\to X$ be a limit minimal geodesic of unit speed,
and set 
\begin{align*}
D\ =\ 2\max_{s\in [0,l]}d(p,\gamma_{\infty}(s)).
\end{align*}
 Now for any $s,t\in
(\delta l,(1-\delta)l)$ ($\delta\in (0,1\slash 10)$ arbitrary) such that
$|s-t|\le C_7(m,D,\delta)l$, and for any $r\in (0,\delta l\slash
10)$, we have
\begin{align*}
d_{GH}(B_X(\gamma_{\infty}(s),r),B_X(\gamma_{\infty}(t),r))\ \le\ 
C_8(n,D,\delta)\left(\frac{|s-t|}{l}\right)^{\frac{1}{4m+2}}r.
\end{align*}
Notice that this estimate is scaling invariant, therefore we could push it to
the tangent cone: let $r_i\to 0$ be a sequence of positive numbers that
determines tangent cones $X_{\gamma_{\infty}(s)}$ and $X_{\gamma_{\infty}(t)}$,
then the above inequalities give
\begin{align*}
d_{GH}\left(B_{X_{\gamma_{\infty}(s)}}(o_s,1),
B_{X_{\gamma_{\infty}(t)}}(o_t,1)\right)\ \le\
C_8(m,D,\delta)\left(\frac{|s-t|}{l}\right)^{\frac{1}{4m+2}},
\end{align*}
whence the desired H\"older continuity of tangent cones.
\end{proof}

\begin{remark}\label{rmk: closedness}
Especially, we see that for $t_i\to t\in (0,l)$, taking
\begin{align*}
\delta=\frac{1}{2l}\min\{d_{X}(\gamma_{\infty}(t),\gamma_{\infty}(0)),
d_{X}(\gamma_{\infty}(t),\gamma_{\infty}(l)),1\slash 10\},
\end{align*}
the above estimate gives:
\begin{align*}
X_{\gamma_{\infty}(t_i)}\ \longright{pointed-Gromov-Hausdorff}\
X_{\gamma_{\infty}(t)},
\end{align*} 
and if $X_{\gamma_{\infty}(t_i)}=\mathbb{R}^k$, then so is $X_{\gamma}(t)$:
$\gamma_{\infty}((0,l))\cap \mathcal{R}_{k}$ is a closed subset of
$\gamma_{\infty}((0,l))$.
\end{remark}

Given the extension lemma (Lemma~\ref{lem: extension}), the H\"older continuity
of tangent cones (Proposition~\ref{prop: tangent_cone_continuity}), and the
$\nu_{\infty}$-negligibility of the singular set (Proposition~\ref{prop:
zero_singular}), we could now prove Theorem~\ref{thm: main3} in a way identical to the original one
in~\cite[Sections 1.2 and 1.4]{CN12}. For the sake of simplicity, we will not
repeat the argument here, but refer the readers to the proofs of Theorems 1.7,
1.18 and 1.20 in~\cite{CN12}.

If we consider the sub-collection $\mathcal{M}_m(F,K;V_0)$, then together with
Theorem~\ref{thmin:a}, we have
\begin{thm}\label{thm: main_N}
Let a sequence $\{(M_i,p_i,g_i,f_i)\}\subset \mathcal N_m(F,K;V_0)$ 
converge to a limit metric space $(X,p_{\infty},d_{\infty},f_{\infty})$ in the
pointed-Gromov-Hausdorff topology, then the regular part $\mathcal{R}\subset X$
is a strongly convex open set, equipped with a limit $C^{1,\alpha}$ metric
$g_{\infty}$ such that $(\mathcal{R},g_{\infty})$ becomes a metric subspace of
$(X,d_{\infty})$.
\end{thm}
\begin{proof}
It has already proven in Theorem \ref{thmin:a} that the regular
part $\mathcal{R}$ is open in $X$, that the convergence is $C^{1,\alpha}$ on
$\mathcal{R}$, and that the limit metrics (in the metric sense and in the
tensor sense) coincide. We only need to prove the strong convexity. Now if a
minimal geodesic $\gamma:[0,1]\to X$ intersects $\mathcal{R}$ non-trivially,
then set
\begin{align*}
I_{\gamma,\mathcal{R}}\ :=\ \{t\in [0,1]:\ \gamma(t)\in \mathcal{R}\}
\end{align*}
is non-empty and is open relative to $[0,1]$. But by Remark~\ref{rmk:
closedness}, we know that $I_{\gamma,\mathcal{R}}$ is also closed relative to
$(0,1)$, therefore $I_{\gamma,\mathcal{R}}=(0,1)$ and therefore the entire
interior of $\gamma$ is contained $\mathcal{R}$. The strong convexity is thus
proven.
\end{proof}

If we further restrict our attention to the collection of all $m$-dimensional
complete Ricci shrinkers, then we have the following
\begin{thm}[Regular-convexity of Gromov-Hausdorff limits]\label{thm: RS}
Let $\{(M_i,p_i,g_i,f_i)\}$ be a sequence of pointed $m$-dimensional Ricci
shrinkers which converges to a limit metric space
$(X,p_{\infty},d_{\infty},f_{\infty})$ equipped with a limit potential function
$f_{\infty}$, in the pointed-Gromov-Hausdorff topology, such that the
associated probability measures $\rho_i$ converges to $\rho_{\infty}$ on $X$,
then the following holds:
\begin{enumerate}
  \item Assuming the sequence is contained in $\mathcal{M}_m(A)$ for some fixed
  positive constant $A$, then the Hausdorff dimension of $X$ is $m$, and
  $\mathcal{R}\subset X$ is a strongly convex open set, which, when equipped
  with the limit metric $d_{\infty}$, becomes an $m$-dimensional Riemannian
  manifold with a $C^{\infty}$ metric tensor that satisfies the Ricci shrinker equation;
  \item Without assuming a uniform positive lower bound of the
  set $\{\mu_{f_i}(M_i)\}$, then there is a unique natural number $k\le
  m$, such that $\rho_{\infty}(X\backslash \mathcal{R}_k)=0$; moreover,
  $\mathcal{R}_k$ is both $\rho_{\infty}$-a.e. convex and weakly convex.
\end{enumerate}
\end{thm}
Clearly, after applying a usual elliptic regularity argument, the first
alternative in this theorem is a special case of Theorem~\ref{thm: main_N}. Moreover, by the
equivalence of the uniform $\boldsymbol{\mu}$-entropy lower bound and the uniform volume
non-collapsing property (see~\cite{Pe1}, ~\cite{HM11} and \cite[Lemma
2.5]{LLW17}), the first alternative of this theorem states the same as
Theorem~\ref{thm: main2}. The second alternative is a restatement of
Theorem~\ref{thm: main3} for the special case of Ricci shrinkers.

\section{Discussion}

In geometric analysis, the compactness of the moduli of certain collection of
spaces, in an appropriate topology, is a fundamental problem. In the setting
of Ricci shrinkers, we would like to ask whether the collection of all conifold
Ricci shrinkers with a given dimension and a uniform lower bound of the
$\boldsymbol{\mu}$-entropy is compact in the
pointed-$\hat{C}^{\infty}$-Cheeger-Gromov topology.

This question is not a direct consequence of Theorem~\ref{thm: main2}, since it
is not true that all conifold Ricci shrinkers arise as the
pointed-$\hat{C}^{\infty}$-Cheeger-Gromov limits of elements in
$\mathcal{M}_m(A)$. In the K\"ahler setting, orbifold K\"ahler-Ricci solitons
(of complex dimension at least $2$) whose quotient singularities are of real codimension at least $4$ and non-smoothable provide
examples of conifold Ricci shrinkers not in the closure of $\mathcal{KM}_{n}(A)$, the moduli space of (complex) $n$-dimensional K\"ahler-Ricci shrinkers with $\boldsymbol{\mu}$-entropy bounded below by $-A$. 
The Riemannian setting is even more complicated. It is by itself an interesting
problem to understand those conifold Ricci shrinkers that are not on the
boundary of $\mathcal{M}_m(A)$, and partial progress towards this direction has
already been made in~\cite{LWs1} and~\cite{LLW17}.

We believe that the compactness question could be answered affirmatively,
in view of the previous work done in the K\"ahler-Ricci flat setting \cite[Theorem
1.3]{CW17A}, especially considering that many of the analytical tools developed
in \cite{CW17A} only assume the Riemannian setting.

\textbf{Acknowledgement.} We would like to thank the anonymous referees for
several valuable comments that help improve the exposition of the paper. The
third-named author was partially supported by NSF grant DMS-1510401, as well as
the General Program of the National Natural Science Foundation of China (Grant
No. 11971452).

\vskip20pt
Shaosai Huang, Department of Mathematics, University of
Wisconsin - Madison; Address: 480 Lincoln Drive, Madison, WI 53706, U.S.A.;
E-mail: sshuang@math.wisc.edu.\\

Yu Li, Department of Mathematics, Stony Brook University; Address: 100 Nicolls
Road, Stony Brook, NY 11794, U.S.A.; E-mail: yu.li.4@stonybrook.edu.\\

Bing  Wang, Institute of Geometry and Physics, 
and Wu Wen-Tsun Key Laboratory of Mathematics,  School of
Mathematical Sciences, University of Science and Technology of China; Address: No. 96 Jinzhai Road,
Hefei, Anhui Province, 230026, China; E-mail: topspin@ustc.edu.cn.


\begin{thebibliography}{99}

\bibitem{AG90} U. Abresch and D. Gromoll, \emph{On complete manifolds with
nonnegative Ricci curvature}, J. Amer. Math. Soc. 3 (1990), no. 2, 355-374.

\bibitem{AGS14} L. Ambrosio, N. Gigli, and G. Savare, \emph{Metric measure
spaces with Riemannian Ricci curvature bounded from below}, Duke Math. J. 163
(2014), no. 7, 1405-1490.


 \bibitem{CCZ08}	H.-D. Cao, B.-L. Chen and X. Zhu,	\emph{Recent developments
 on Hamilton's Ricci flow}, Surveys in differential geometry, Vol. XII. Geometric
 flows, 47-112, Surv. Diff. Geom., 12, Int. Press, Somerville, MA, 2008.

\bibitem{CZ10} H.-D. Cao and D. Zhou, \emph{On complete gradient shrinking
Ricci solitons}, J. Differential Geom. 85 (2010), no. 2, 175-186.

 \bibitem{ChCo96} J. Cheeger and T. H. Colding, \emph{Lower Bounds on Ricci
 Curvature and the Almost Rigidity of Warped Products}, Ann. of Math. (2) 144
 (1996), no. 1, 189-237.
   
   \bibitem{ChCo97} J. Cheeger and T. H. Colding, \emph{On the structure
   of spaces with Ricci curvature bounded below. I}, J. Differential Geom. 45
   (1997), no. 3, 406-480.

 \bibitem{ChCo00} J. Cheeger and T. H. Colding, \emph{On the structure of
 spaces with Ricci curvature bounded below. III}, J. Differential Geom. 54
 (2000), no. 1, 37-74.

  \bibitem{BLC} B.-L. Chen, \emph{Strong uniqueness of the Ricci flow}, J.
   Differential Geom. 82 (2009), no. 2, 362-382.
   
   \bibitem{CW17A} X. Chen and B. Wang, \emph{Space of Ricci flows (II)---Part
   A:  moduli of singular Calabi-Yau spaces}, Forum Math. Sigma 5 (2017), e32,
   103 pp.
   
    \bibitem{CW17B} X. Chen and B. Wang, \emph{Space of Ricci flows (II)---Part
    B: weak compactness of the flows}, arXiv: 1405.6797, to appear in J.
    Differential Geom.

\bibitem{CN12}  T. H. Colding and A. Naber, \emph{Sharp H\"older continuity of
tangent cones for spaces with a lower Ricci curvature bound and applications},
Ann. of Math. (2) 176 (2012), no. 2, 1173-1229.


 
\bibitem{GMS}  N. Gigli, A. Mondino, G. Savar\'{e}, \emph{Convergence of pointed 
non-compact metric measure spaces and stability of Ricci curvature bounds and
heat flows}, Proc. Lond. Math. Soc. (3) 111 (2015), no. 5, 1071-1129.

\bibitem{Gromov} M. Gromov, \emph{Metric structures for Riemannian and
non-Riemannian spaces}, Progress in Mathematics, 152. Birkh\"auser Boston, Inc.,
Boston, MA, 1999. xx+585 pp. ISBN: 0-8176-3898-9

 \bibitem{Gr09} A. Grigor'yan, \emph{Heat Kernel and Analysis on Manifolds},
 AMS\slash IP Studies in Advanced Mathematics, 47. American Mathematical
 Society, Providence, RI; International Press, Boston, MA, 2009. xviii+482 pp.
 ISBN: 978-0-8218-4935-4 
 
 \bibitem{Ham95} R. S. Hamilton, \emph{The formation of singularities in the
 Ricci flow},  Surveys in differential geometry, Vol. II (Cambridge, MA, 1993),
 7-136, Int. Press, Cambridge, MA, 1995.


     
 \bibitem{HM11} R. Haslhofer and R. M\"{u}ller,	\emph{A compactness theorem for
 complete Ricci shrinkers},	Geom. Funct. Anal. 21 (2011), no. 5, 1091-1116.
 
 \bibitem{HM14} R. Haslhofer and R. M\"{u}ller, \emph{A note on the compactness
 theorem for 4d Ricci shrinkers}, Proc. Amer. Math. Soc. 143 (2015), no. 10, 
 4433-4437.

\bibitem{H11}  S. Honda, \emph{Bishop-Gromov type inequality on Ricci limit
spaces}, J. Math. Soc. Japan 63 (2011), no. 2, 419-442.


\bibitem{Huang17} S. Huang, \emph{$\varepsilon$-Regularity and structure
of four-dimensional shrinking Ricci solitons}, Int. Math. Res. Not. IMRN 2020,
no. 5, 1511-1574.

\bibitem{HW20a} S. Huang and B. Wang, \emph{Rigidity of the first Betti number
via Ricci flow smoothing}, arXiv: 2004.09762. 

\bibitem{K16} N. N. Khanh, \emph{Gradient estimates of Li Yau type for a general
heat equation on Riemannian manifolds}, Arch. Math. (Brno) 52 (2016), no. 4,
207-219.

  \bibitem{Naber10}  A. Naber, \emph{Noncompact shrinking four solitons
  with nonnegative curvature},  J. Reine Angew. Math. 645 (2010), 125-153.

  \bibitem{NW08}  L. Ni and N. Wallach, \emph{On a classification of
  gradient shrinking solitons}, Math. Res. Lett. 15 (2008), no. 5, 941-955.


 \bibitem{LLW17}  H. Li, Y. Li and B. Wang, \emph{On the structure of Ricci
 shrinkers}, arXiv: 1809.04049.

\bibitem{LY86} P. Li and S.-T. Yau, \emph{On the parabolic kernel of the
Schr\"odinger operator}, Acta Math. 156 (1986), no. 3-4, 153-201.

\bibitem{LWs1}	Y. Li and B. Wang, \emph{The rigidity of Ricci shrinkers of
dimension four},  Trans. Amer. Math. Soc. 371 (2019), no. 10, 6949-6972.

\bibitem{LV09} J. Lott and C. Villani, \emph{Ricci curvature for metric-measure
spaces via optimal transport}, Ann. of Math. (2) 169 (2009), no. 3, 903-991.
 	
 \bibitem{Pe1}  G. Perelman, \emph{The entropy formula for the Ricci flow and
 its geometric applications}, arXiv: math.DG\slash 0211159.

\bibitem{Pe2} G. Perelman,	\emph{Ricci flow with surgery on three-manifolds},
 arXiv: math.DG\slash 0303109.


\bibitem{LSC92} L. Saloff-Coste, \emph{A note on Poincar\'e, Sobolev and
Harnack inequality}, Int. Math. Res. Not. 1992, no. 2, 27-38.


\bibitem{KTS96} K.-T. Sturm, \emph{Analysis on local Dirichlet spaces. III. The
parabolic Harnack inequality}, J. Math. Pures Appl. (9) 75 (1996), no. 3,
273-297.

\bibitem{KT06a} K.-T. Sturm, \emph{On the geometry of metric measure spaces. I},
Acta Math. 196 (2006), no. 1, 65-131.

\bibitem{KT06b} K.-T. Sturm, \emph{On the geometry of metric measure spaces.
II}, Acta Math. 196 (2006), no. 1, 133-177.

\bibitem{Vil08} C. Villani, \emph{Optimal Transport, Old and New},  Grundlehren
der Mathematischen Wissenschaften [Fundamental Principles of Mathematical
Sciences], 338. Springer-Verlag, Berlin, 2009. xxii+973 pp. ISBN:
978-3-540-71049-3

\bibitem{WZ17} F. Wang and X. Zhu, \emph{The structure of spaces with
Bakry-\'Emery Ricci curvature bounded below}, J. Reine Angew. Math. 757 (2019),
1-50.

\bibitem{WW09} G. Wei and W. Wylie, \emph{Comparison geometry for the
Bakry-\'Emery Ricci tensor}, J. Differential Geom. 83 (2009), no. 2, 377-406.

\bibitem{WW16} J.-Y. Wu and P. Wu, \emph{Heat kernel on smooth metric measure
spaces and applications}, Math. Ann. 365 (2016), no. 1-2, 309-344.

\bibitem{ZZ17} Q. S. Zhang and M. Zhu, \emph{Bounds on harmonic radius and
limits of manifolds with bounded Bakry-\'Emery Ricci curvature}, J. Geom. Anal.
29 (2019), no. 3, 2082-2123.

\bibitem{ZLZ10} Z. Zhang, \emph{Degeneration of shrinking Ricci
solitons}, Int. Math. Res. Not. 2010, no. 21, 4137-4158.

\bibitem{Shunhui} S.-H. Zhu, \emph{The comparison geometry of Ricci curvature.}
Comparison geometry (Berkeley, CA, 1993-94), 221-262, Math. Sci. Res. Inst.
Publ., 30, Cambridge Univ. Press, Cambridge, 1997.
\end{thebibliography}
\end{document}